\documentclass[reqno]{amsart}
\pdfoutput=1
\usepackage{fullpage}
\usepackage{amssymb,amsthm,amsmath,graphicx,pinlabel,wasysym,tikz-cd}
\usepackage[colorlinks]{hyperref}
\usepackage[style=default, margin=0pt, parskip=0pt, hangindent=0pt, indention=0pt, singlelinecheck=true, labelfont=up, justification=justified]{subcaption}

\DeclareMathOperator{\crit}{crit}

\DeclareMathOperator{\id}{id}

\DeclareMathOperator{\gr}{gr}

\newcommand{\A}{\mathcal{A}}

\newcommand{\co}{\colon\thinspace}

\newcommand{\R}{\mathbb{R}}
\newcommand{\Z}{\mathbb{Z}}

\newcommand{\CP}{\mathbb{CP}{}^{2}}
\newcommand{\CPb}{\overline{\mathbb{CP}}{}^{2}}
\newcommand{\MYhref}[3][blue]{\href{#2}{\color{#1}{#3}}}

\newcommand{\s}{\mathcal{S}}
\newcommand{\7}{E(1)_{7_6}}
\newcommand{\1}{E(1)_{10_{133}}}
\newcommand{\x}{E(1)_K}
\newcommand{\y}{E(1)_{K'}}

\theoremstyle{definition}
\newtheorem{thm}{Theorem}[section]
\newtheorem{lemma}[thm]{Lemma}

\newtheorem{prop}[thm]{Proposition}

\newtheorem{df}[thm]{Definition}

\newtheorem{q}[thm]{Question}

\title{The salient crossings of a crown diagram}
\author[Jonathan D. Williams]{Jonathan D. Williams} 
\address{Department of Mathematical Sciences, Binghamton University}
\email{\MYhref{mailto:jdw.math@gmail.com}{jdw.math@gmail.com}}

\begin{document}
\begin{abstract}A crown diagram of a smooth, closed oriented 4-manifold can be thought of as the projection of a link in the product of a closed surface and the circle, with chords in the circle direction connecting the strands of each crossing. This paper uses a straightforward assignment of integers to these chords to show that the smooth structures on the topological 4-manifold underlying a pair of Fintushel-Stern knot surgery 4-manifolds which are known to have the same Seiberg-Witten invariant are not isotopic. A natural question is to determine if the argument may be strengthened to show these manifolds are not diffeomorphic.\end{abstract}
\maketitle
\section{Introduction}\label{introduction}
Let $M$ be a smooth, closed, oriented, connected 4-manifold, and let $\Sigma$ be a closed oriented surface of genus $g>2$. The construction of this paper begins with a certain type of smooth map $\alpha\co M\to S^2$ of which $\Sigma$ is a regular fiber, a \emph{crown map}, discussed in Section \ref{sddefs}. The surface $\Sigma$ is decorated with a cyclically ordered collection of simple closed curves $\Gamma=\{\gamma_n\}_{n\in\Z/k\Z}$ coming from $f$, which has a natural interpretation as a $k$-component link $L$ in $S^1\times\Sigma$ using the natural embedding of additive groups $\Z/k\Z\to S^1$, $n\mapsto\theta_n\in S^1$:
\[L=\left\{\{\theta_n\}\times\gamma_n\ |\ n\in\Z/k\Z\right\}.\]
The condition that each component of $L$ inhabits its own slice $\{\theta_n\}\times\Sigma$ is a rather useful one; such links are called \emph{cyclically stacked} links in this paper. Analogously, a \emph{linearly stacked} link lies in $[0,1]\times\Sigma$, with each component inhabiting its own subset $\{t\}\times\Sigma$. There are a few interesting sources of stacked links, such as Heegaard diagrams of closed 3-manifolds, yielding links which are linearly stacked with two levels of $g$-component links in $[0,1]\times\Sigma$ (or $2g$ levels, after a slight perturbation), and trisection diagrams of smooth closed 4-manifolds, yielding cyclically stacked links with three levels in $S^1\times\Sigma$. These links, however, do not present an obvious distinguished collection of crossings as crown diagrams do. For crown diagrams, this collection leads to a basic smooth 4-manifold invariant, the collection of so-called \emph{salient sequences} defined in Section \ref{chorddefs} and established as an invariant in Section \ref{invariance}. Section \ref{examples} uses this to show that distinguish the smooth structures of the knot surgery manifolds $\7$ and $\1$ cannot be isotopic in the sense of smoothing theory, where two smooth manifolds $A$ and $B$ are isotopic smoothings of the topological 4-manifold $X$ if there is a smooth manifold $Y$ and a homeomorphism to $X\times[0,1]\to Y$ such that the projection to $[0,1]$ is regular, and diffeomorphisms sending the two ends of $Y$ to $A$ and $B$. Supposing $\x$ and $\y$ are diffeomorphic, the author knows of no other way to distinguish the isotopy classes of their smooth structures.

\subsection*{Acknowledgements}The author would like to thank Jos\'{e} Rom\`{a}n Aranda Cuevas, Yasha Eliashberg, Bob Gompf, Rob Kirby, Robert Lewis, Robert Lipshitz, Cary Malkiewich, Jeff Nye, Tim Perutz, and Lisa Piccirillo for helpful conversations during the course of this work. The author would especially like to thank Umur \c{C}ift\c{c}i for his crucial contribution of the software EQMaker.

\section{Definitions}\label{definitions}
\subsection{Topological setting}\label{sddefs}
A \emph{Morse 2-function} is a smooth map $\alpha\co M\to S$, where $S$ is a smooth 2-manifold, such that each critical point of $f$ is locally modeled on a \emph{fold} critical point,\begin{equation}\label{foldeq}(x_1,x_2,x_3,x_4)\mapsto(x_1,x_2^2+x_3^2\pm x_4^2),\end{equation} or a \emph{cusp} critical point,\begin{equation}\label{cuspeq}(x_1,x_2,x_3,x_4)\mapsto(x_1,x_2^3-3x_1x_2+x_3^2\pm x_4^2).\end{equation}
The critical locus $\crit f$ forms a smooth 1-submanifold of $M$, stratified into a 1-dimensional stratum of fold points and a finite stratum of cusp points. Crown diagrams come from what this paper calls \emph{crown maps}, after a suggestion of R. Kirby (In previous work such as \cite{B,BH,W1,W2}, a crown diagram was called a \emph{surface diagram}). A crown map is a Morse 2-function $f$ for which $S=S^2$, the signs above are both chosen to be negative (that is, the critical points are \emph{indefinite} folds and cusps), and $\crit f$ is a single circle on which $f$ is injective (Figure \ref{crown}). A crown map induces a singular fiber bundle called a \emph{crown fibration} on $M$ as follows. The critical values of $f$ partition $S^2$ into a pair of disks, the \emph{higher-genus side} whose regular fibers are diffeomorphic to a closed oriented genus-$g$ surface $\Sigma$, and the \emph{lower-genus side}, which has closed oriented genus $g-1$ fibers. The two disks meet along the critical image of $f$, which is a circle with corners, where each corner is the image of a cusp point.
\begin{figure}
	\centering
	\begin{subfigure}{0.30\textwidth}
		\includegraphics{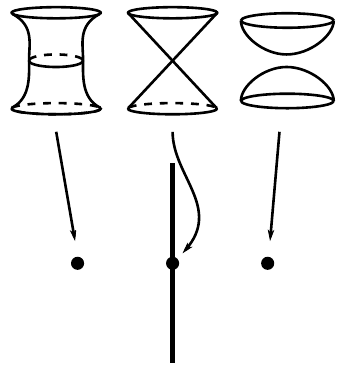}
		\caption{The image of a fold arc.}
		\label{fold}
	\end{subfigure}\qquad
	\begin{subfigure}{0.30\textwidth}
		\includegraphics{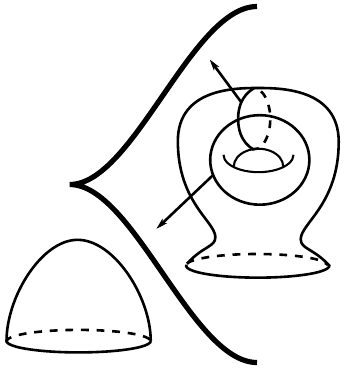}
		\caption{The image of a cusp point connecting two fold arcs.}
		\label{cusp}
	\end{subfigure}\qquad
	\begin{subfigure}{0.30\textwidth}
		\labellist
		\small\hair 2pt
		\pinlabel $v_3$ at 67 80
		\pinlabel $v_2$ at 82 53
		\pinlabel $v_1$ at 68 25
		\pinlabel $v_k$ at 34 25
		\pinlabel $v_4$ at 32 80
		\pinlabel $v_0$ at 42 53
		\endlabellist
		\includegraphics{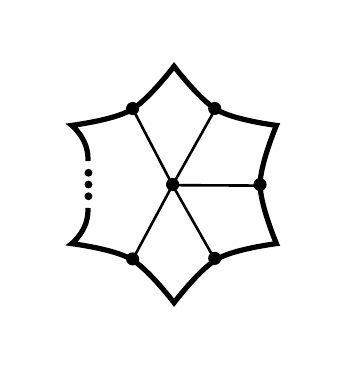}
		\caption{The critical image of a crown map, decorated with a graph $T$.}
		\label{crown}
	\end{subfigure}
	\caption{The critical image of an indefinite Morse 2-function. Arrows and circles in depicted reference fibers indicate the attaching circles corresponding to indicated fold arcs.}\label{crit}
\end{figure}
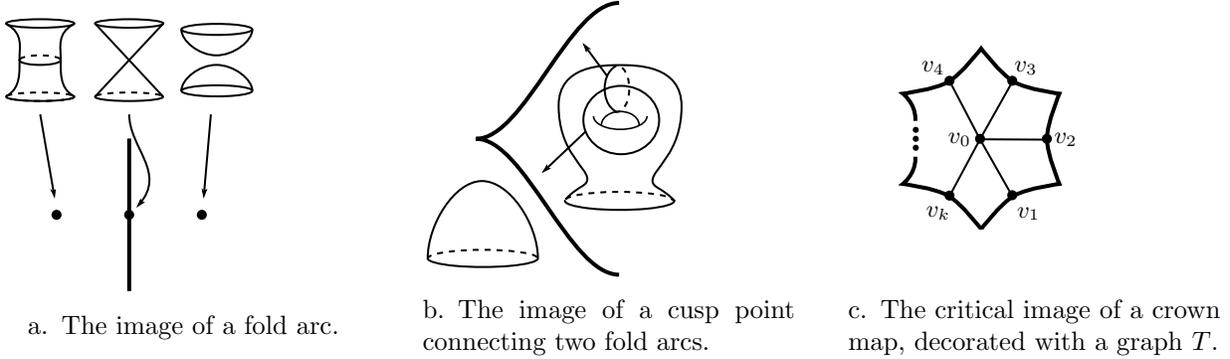

The local model for an indefinite fold point is $(h,\id_\R)\co\R^3\times\R\to\R\times\R$, where $h$ is the local model for a 3-dimensional Morse critical point of index 1 or 2 (in Figure \ref{fold}, it is index 1 with $h$ understood to increase from left to right along any horizontal arc). Such a Morse function induces a fibration of $R^3$ by the one- and two-sheeted hyperboloids $x^2+y^2-z^2=const$, and each one-sheeted hyperboloid $\{x^2+y^2-z^2=const>0\}$ intersects the $xy$ plane in a distinguished circle. In the image of a Morse 2-function, the region of regular values corresponding to $const>0$ is called the \emph{higher genus side} of the fold arc and the distinguished circle is a simple closed curve in the regular fiber called a \emph{vanishing cycle}. All of the vanishing cycles are depicted as simple closed curves in a fixed reference fiber $f^{-1}(p)$, measured by choosing a reference path from $p$ to the image of a fold arc, and a connection for the fibration, then marking the points in the regular fiber that flow  to the fold point. In general, the isotopy class of the vanishing cycle depends on these choices; however, the crown condition implies that the vanishing cycles in this paper are all non-separating and lie in a smooth, closed oriented surface $\Sigma$ of genus at least 3, and different choices of higher-genus reference fiber, connection and reference path yield isotopic vanishing cycles \cite[Section 4]{BH}. 

With this understood, a \emph{crown diagram} of the smooth, closed oriented 4-manifold $M$ is the oriented diffeomorphism class of a pair $(\Sigma,\Gamma)$, where $\Sigma$ is a closed, oriented surface of genus greater than 2 and $\Gamma$ is a collection of simple closed curves in $\Sigma$ which occurs as the counter-clockwise ordered sequence of vanishing cycles of a crown fibration whose oriented higher-genus fiber is $\Sigma$. A diffeomorphism of crown diagrams is required to preserve the cyclic ordering of $\Gamma$.  According to \cite[Corollary 2]{W1}, a crown diagram specifies $M$ up to orientation-preserving diffeomorphism.

The local model for an indefinite cusp is the same as the local model for the introduction via homotopy of a canceling pair of index-1 and index-2 Morse critical points in a 3-manifold (the homotopy parameter would increase from left to right in Figure \ref{cusp}). For this reason, the vanishing cycles of the two fold arcs which meet at a cusp point can be depicted to have exactly one transverse crossing in $\Sigma$. \begin{df}After canceling all crossings but one from each pair of consecutive vanishing cycles in a crown diagram, each unique transverse crossing between consecutive vanishing cycles is called a \emph{salient crossing}.\end{df}
At various points it will be useful to designate a vanishing cycle as the ``first" element of $\Gamma$, indexing with integers as $(\gamma_1,\ldots,\gamma_k)$, where the ordering of any three consecutive elements respects the cyclic order of $\Gamma$. Such a choice will be called a \emph{labeling} of $\Gamma$.

\begin{df}A \emph{homotopy class} of $(\Sigma,\Gamma)$ is the smooth homotopy class of any crown map whose diagram is $(\Sigma,\Gamma)$ for some choice of reference fiber, connection and reference paths.\end{df}

It may be possible (and would be interesting to find) that two non-homotopic crown maps $M\to S^2$ have the same crown diagram. In that case, both homotopy classes of maps would be homotopy classes of the diagram.

\begin{df}Two crown diagrams $(\Sigma_i,\Gamma_i)_{i=1,2}$ are \emph{isotopic} if there is an isotopy of individual vanishing cycles in $\Gamma_1$ resulting in a crown diagram which is diffeomorphic to $(\Sigma_2,\Gamma_2)$. \end{df}

\begin{df}A crown diagram $(\Sigma,\Gamma)$ is \emph{allowable} if $\Sigma\setminus\Gamma$ is a disjoint union of open disks whose closures are closed disks, and $|\gamma_i\cap\gamma_{i+1}|=1$ for every $i$.\end{df}
\begin{prop}Every crown diagram with four or more vanishing cycles is isotopic to an allowable crown diagram, and two isotopic allowable crown diagrams can be made diffeomorphic using an isotopy consisting of moves taken from Figure \ref{isotfig}, such that each move preserves allowability.\end{prop}
\begin{proof}One may apply what are generally called \emph{finger moves} or point-pushing maps to individual vanishing cycles to subdivide regions until $(\Sigma,\Gamma)$ is allowable, introducing crossings between non-consecutive vanishing cycles. To convert into a different yet isotopic allowable diagram, one may add temporary finger moves at various stages of the isotopy as needed to preserve allowability at every stage.\end{proof}

\subsection{The chords of a crown diagram}\label{chorddefs}
The cyclic ordering of $\Gamma$ gives data reminiscent of crossings in knot diagrams as follows. For $i,j\in\Z/k\Z$, an intersection $x\in\gamma_i\cap\gamma_j$ is associated with the two arcs $[\theta_i,\theta_j]\times\{x\}$ and $[\theta_j,\theta_i]\times\{x\}$ whose union is $S^1\times\{x\}$. These are both oriented as subspaces of $S^1\times\Sigma$, connecting the circles $\{\theta_i\}\times\gamma_i$ and $\{\theta_j\}\times\gamma_j$. These arcs are called the \emph{chords} of $x$.

Any chord $[\theta_i,\theta_j]\times\{x\}$ singles out the sublink of $L$ contained in $[\theta_i,\theta_j]\times\Sigma$ and the complimentary chord $[\theta_j,\theta_i]\times\{x\}$ singles out the sublink of $L$ contained in $[\theta_j,\theta_i]\times\Sigma$, because the interiors of the two intervals are disjoint positively oriented arcs in $S^1$. The crossing $x$ appears in both sublinks, and if $\gamma_i$ passes over $\gamma_j$ in one, then $\gamma_i$ passes under $\gamma_j$ in the other
. In either case, given a chord $a$, one may choose a labeling of $\Gamma$ so that the under-crossing vanishing cycle at $x$ is $\gamma_1$. Then the over-crossing vanishing cycle would be $\gamma_n$ for some $n\in\{2,\ldots,k\}$ and that corresponding sublink is \[L_a=\bigcup\limits_{i=1}^n\{\theta_i\}\times\gamma_i\subset[\theta_1,\theta_n]\times\Sigma.\]For each labeling, the two strands of each crossing in the crown diagram necessarily come from an intersection of two distinct circles, so this paper adopts the visual convention that, having chosen a crossing and a labeling, the under-crossing strand is the one with the lesser subscript. Using this convention, the corners of the disks adjacent to each crossing in a sublink $L_a$ can be given signs as in Figure \ref{crossing1}.

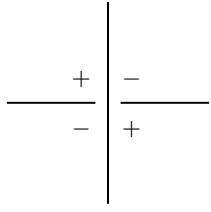
\begin{figure}\centering
	\labellist
	\small\hair 2pt
	\pinlabel $-$ at 47 47
	\pinlabel $-$ at 28 28
	\pinlabel $+$ at 47 28
	\pinlabel $+$ at 28 47
	\endlabellist
	\includegraphics{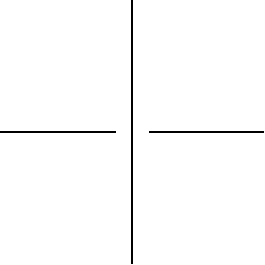}
	\caption{The signs of the quadrants near a crossing of vanishing cycles in a crown diagram are defined as in this figure only after choosing a labeling for $\Gamma$.}
	\label{crossing1}
\end{figure}
Since $\gamma_i$ intersects $\gamma_{i+1}$ at a unique point for each $i\in\Z/k\Z$, there are $k$ distinguished chords whose associated sublinks each have two components $\gamma_i,\gamma_{i+1}$ and one crossing, which is a salient crossing. These chords are the \emph{salient chords} of $\A(\Sigma,\Gamma)$. 

Suppose $(\Sigma,\Gamma)$ is allowable, choose a labeling of $\Gamma$ and let $U$ be the collection of disks forming the disjoint union $\Sigma\setminus\cup_i\gamma_i=\sqcup D^2$. For any crossing there is a unique chord which does not intersect $[\theta_k,\theta_1]\times\Sigma$; in other words, it is the chord which is contained in the link diagram whose over-crossing strands are designated by the labeling. Thus, a labeling gives signs to all the corners in the diagram. Given a labeling and a disk $u\in U$, denote by $Q_+(u)$ the set of chords corresponding to positive corners of $u$ and denote by $Q_-(u)$ the set of chords corresponding to negative corners of $u$. Note that allowability implies $Q_+(u)$ and $Q_-(u)$ are disjoint. Assign a variable $\gr(a)$ to each chord $a$ and impose the following \emph{grading equations} for every labeling of $\Gamma$ and for every region $u$:
\begin{equation}\sum\limits_{a\in Q_+(u)}\gr(a)-\sum\limits_{a\in Q_-(u)}\gr(a)=|Q_+(u)|-|Q_-(u)|.\label{dimformula}\end{equation}
The system of equations given by Equation \ref{dimformula} will be referred to as the \emph{grading system}. It is defined as a collection of equations associated to disks, but Equation \ref{dimformula} also makes sense for any region $u\subseteq\Sigma$ whose boundary is a union of arcs of vanishing cycles that meet at crossings. The right hand side of Equation \ref{dimformula} assigns an integer to each $u\in U$, and it is not hard to check that this assignment is additive in the sense that the grading equation for a union of disks is the sum of the grading equations for the disks.

It is now possible to define the main invariant of the paper.
\begin{df}A \emph{salient sequence} of $(\Sigma,\Gamma)$ is a choice of values of $gr(x_i)$ for the salient chords $x_1,\ldots,x_k$ of the diagram, obtained from a solution of the collection of all the grading equations for the disks in the complement of the vanishing cycles. Let $\s(\Sigma,\Gamma)$ denote the set of salient sequences for the diagram $(\Sigma,\Gamma)$.\end{df}
The grading system is defined as a system of equations over the integers, but it may only have solutions in a finite cyclic group, and this solution is not necessarily unique.
\section{Invariance}\label{invariance}
This section contains the proof that $\s(\Sigma,\Gamma)$ is an invariant of the collection of crown diagrams of $M$ with any fixed genus and homotopy class.
\begin{df}For two closed curves $a,b$ in $\Sigma$, a \emph{slide} of $a$ over $b$ is the result of replacing $a$ with the connect sum $a\# b$ using some choice of connecting arc, which is a smooth embedded curve from one point of $a$ to one point of $b$ whose interior is disjoint from $a$ and $b$. In this situation, $a$ is the \emph{nonstationary} circle and $b$ is the \emph{stationary} circle.\end{df}

\begin{thm}\label{invthm}If two allowable crown diagrams $(\Sigma_i,\Gamma_i)_{i=1,2}$ of $M$ have the same genus and a common homotopy class, then $\s(\Sigma_1,\Gamma_1)=\s(\Sigma_2,\Gamma_2)$.\end{thm}

\begin{lemma}\label{isotopy}If two crown diagrams $(\Sigma_i,\Gamma_i)_{i=1,2}$ of $M$ are isotopic and allowable, then $\s(\Sigma_1,\Gamma_1)=\s(\Sigma_2,\Gamma_2)$.\end{lemma}
\begin{proof}
For any labeling, any isotopy sending $(\Sigma_1,\Gamma_1)$ to $(\Sigma_2,\Gamma_2)$ can be slightly perturbed to occur as a sequence of Reidemeister moves, chosen from Figure \ref{isotfig}.
\begin{figure}
	\centering
	\begin{subfigure}{0.3\textwidth}
		\labellist
		\small\hair 2pt
		\pinlabel {II {\scalebox{2.5}{$\downarrow$}} {\scalebox{2.5}{$\uparrow$}} II$^{-1}$} at 63 70
		\pinlabel $\bullet\ p_1$ at 63 140
		\pinlabel $\bullet\ p_2$ at 63 115
		\pinlabel $\bullet\ p_3$ at 63 90
		\pinlabel $\bullet\ p_4$ at 63 45
		\pinlabel $\bullet\ p_5$ at 22 26
		\pinlabel $\bullet\ p_6$ at 63 25
		\pinlabel $\bullet\ p_7$ at 102 26
		\pinlabel $\bullet\ p_8$ at 63 5
		\pinlabel $a$ at 43 26
		\pinlabel $b$ at 77 26
		\endlabellist
	\includegraphics[width=\textwidth]{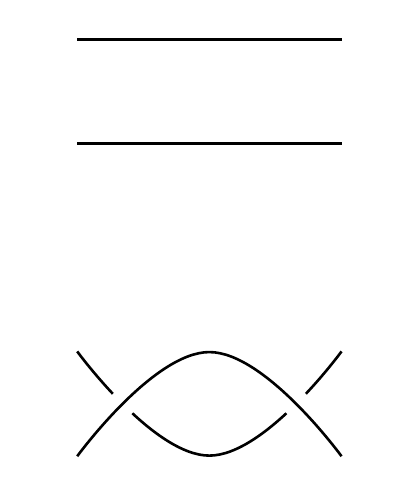}
	\caption{Reidemeister-II.}
	\label{r2}
	\end{subfigure}
	\begin{subfigure}{0.3\textwidth}
		\labellist
		\small\hair 2pt
		\pinlabel {{\scalebox{2.5}{$\downarrow$}}} at 60 70
		\pinlabel $d$ at 58 135
		\pinlabel $c$ at 80 108
		\pinlabel $e$ at 39 108
		\pinlabel $d$ at 63 8
		\pinlabel $e$ at 80 35
		\pinlabel $c$ at 42 35
		\pinlabel $\bullet\ p_t$ at 62 107
		\pinlabel $\bullet\ p_{9}$ at 85 120
		\pinlabel $\bullet\ p_{10}$ at 95 92
		\pinlabel $\bullet\ p_{11}$ at 95 48
		\pinlabel $\bullet\ p_{12}$ at 85 22
		\pinlabel $\bullet\ p_t$ at 62 32
		\endlabellist
	\includegraphics[width=\textwidth]{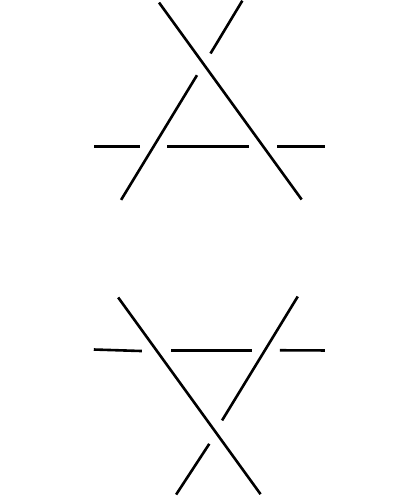}
	\caption{Reidemeister-IIIa.}
	\label{r3a}
	\end{subfigure}
	\begin{subfigure}{0.3\textwidth}
		\labellist
		\small\hair 2pt
		\pinlabel {{\scalebox{2.5}{$\downarrow$}}} at 60 70
		\endlabellist	
	\includegraphics[width=\textwidth]{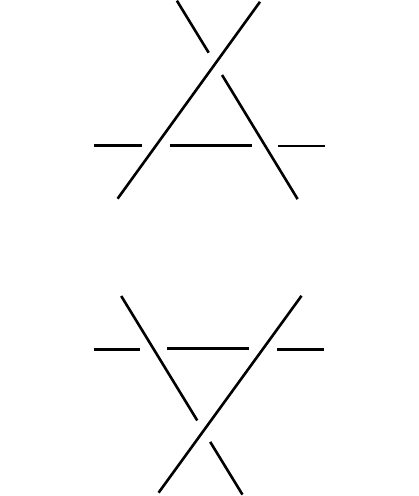}
	\caption{Reidemeister-IIIb.}
	\label{r3b}
	\end{subfigure}
	\caption{Isotopies of $L$ corresponding to isotopies of a crown diagram.}\label{isotfig}
\end{figure}
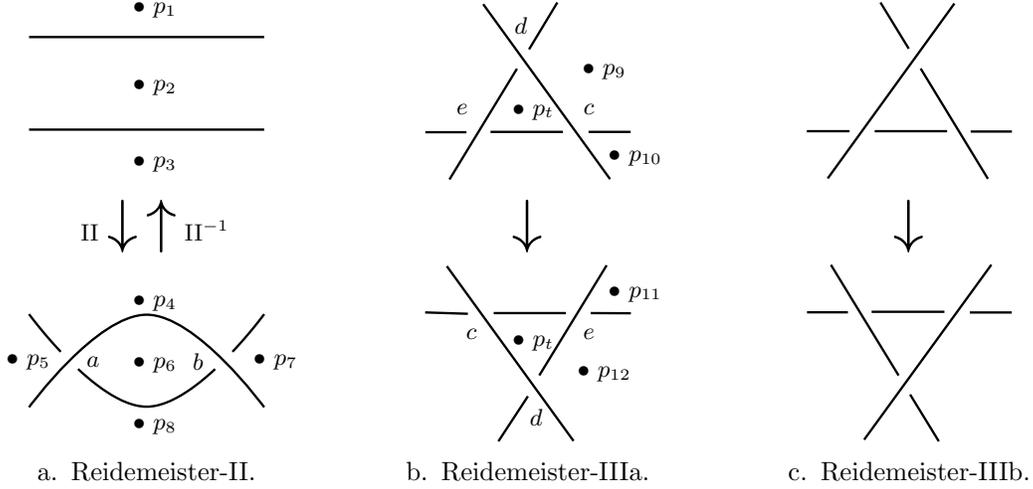
Invariance of $\s$ comes from examining the grading equations for those disks forming $\Sigma\setminus\bigcup_i\gamma_i$ that are affected by such an isotopy, with an arbitrary fixed labeling. Except for the corners $a$ and $b$ in Figure \ref{r2}, the crossings have unchanging labels which follow along with the isotopy, and outside of the pictured regions the disks, their corners, and their grading equations are unchanged. In the following arguments, let $E_i$ denote the grading equation for the disk containing the point $p_i$ in Figure \ref{isotfig}, for any labeling that yields the pictured crossing data. Since relabeling yields the same pictures, possibly rotated, the following arguments are sufficient.

For the Reidemeister-II move in Figure \ref{r2}, $E_1=E_4+E_6$, $E_2=E_5-E_6+E_7$, and $E_3=E_6+E_8$, while all other equations for $(\Sigma_1,\Gamma_1)$ are also equations for $(\Sigma_2,\Gamma_2)$. For this reason, $\s(\Sigma_2,\Gamma_2)\subseteq\s(\Sigma_1,\Gamma_1)$. On the other hand, the chords $c_a$ and $c_b$ corresponding to the crossings $a$ and $b$ in the grading system for $(\Sigma_2,\Gamma_2)$ can be removed from all equations except for that of $u_6$ by rearranging the sums above for $u_4$ and $u_8$, and by substitution for $u_5$ and $u_7$. For $u_5$, for example, the $-gr(c_a)$ appearing in the equation for $u_5$ can be replaced by $gr(b)$ using the equation for $u_6$, and then $gr(b)$ can be replaced by the result of solving for $gr(b)$ in the equation for $u_7$, yielding the equation for $u_2$. The argument for $u_7$ runs similarly. For this reason, appending the two grading equations in four variables coming from $u_6$ to the first grading system results in a grading system which is equivalent to the second, and since the variables for those two extra equations do not appear elsewhere in the first system, $\s(\Sigma_1,\Gamma_1)\subseteq\s(\Sigma_2,\Gamma_2)$. 

For a Reidemeister-II$^{-1}$ move between consecutive vanishing cycles, it may be the case that one of the crossings that disappear is a salient crossing. This cannot happen if both diagrams are allowable, but it is possible that a sequence of slides results in a diagram which is not allowable, but becomes an allowable diagram after canceling some crossings between consecutive vanishing cycles. In that case, the relevant figure is a pair of bigons which share a corner, coming from three crossings between two vanishing cycles. In this case the two grading equations for any chosen labeling set the three associated $gr$ variables to be equal.

For the Reidemeister-III moves, referring to Figure \ref{r3a}, the central triangles before and after a Reidemeister-IIIa move have the same equation $E_t$. The regions opposite a corner of the triangle have equations analogous to $E_9=E_{11}-E_t$, and regions opposite a side of the triangle have equations analogous to $E_{10}=E_{12}+E_t$. There are also analogous equations for the disks in Figure \ref{r3b}. Since the reverse of a Reidemeister-IIIa move is another Reidemeister-IIIa move, and the same is true of a Reidemeister-IIIb move, the systems of grading equations before and after a Reidemeister-IIIa or Reidemeister-IIIb move are equivalent. Since the grading equations are only defined after choosing a labeling, the cases in which the central triangle has all positive or all negative corners do not occur.\end{proof}
According to \cite[Theorem 4.9]{W3}, if two crown diagrams of $M$ have the same genus and homotopy class, then they are related by the so-called \emph{genus-preserving moves}, which are handleslide, multislide, and shift. According to \cite[Theorem 4.3]{W3}, these moves are all realized as sequences of slides of a certain form determined by which move is applied. The following lemmas show that these sequences of slides leave the salient sequence $\s$ unchanged.
\begin{lemma}\label{slides}If two crown diagrams $(\Sigma_i,\Gamma_i)_{i=1,2}$ of $M$ are related by a multislide or a handleslide move, then $\s(\Sigma_1,\Gamma_1)=\s(\Sigma_2,\Gamma_2)$.\end{lemma}
\begin{proof}This follows from two facts. First, for any chord $x$, the set of values taken by $\gr(x)$ in the collection of solutions to the grading system is unchanged by any slide. Second, the handleslide and multislide moves can be performed in such a way that the collection of salient crossings is left unchanged.

\begin{figure}\centering
	\labellist
	\small\hair 2pt
	\pinlabel $\bullet\ p_1$ at 25 62
	\pinlabel $\bullet\ p_2$ at 70 62
	\pinlabel $\bullet\ p_3$ at 270 62
	\pinlabel $\bullet\ p_4$ at 320 78
	\pinlabel $\bullet\ p_5$ at 320 46
	\pinlabel $\bullet\ p_6$ at 375 95
	\pinlabel $\bullet\ p_7$ at 375 28
	\pinlabel {{\scalebox{2.5}{$\rightarrow$}}} at 205 62
	\endlabellist
	\includegraphics{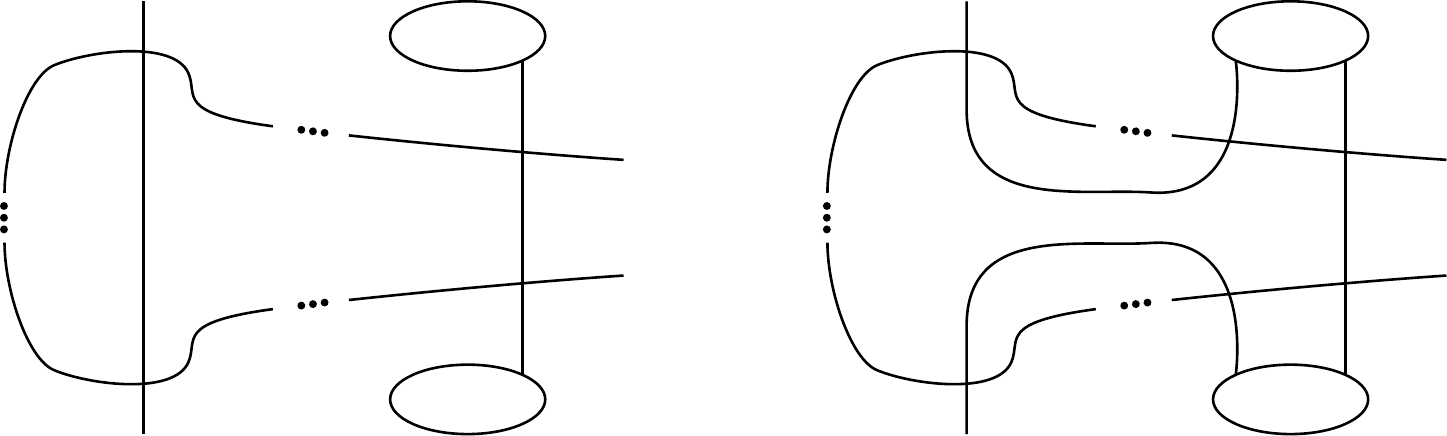}
	\caption{A slide. The vertical arc on the left is part of a vanishing cycle that slides over the vertical vanishing cycle on the right. The dotted parts of the segments may contain crossings. The two ovals in each figure are the ends of an annulus in $\Sigma$.}
	\label{slide}
\end{figure}
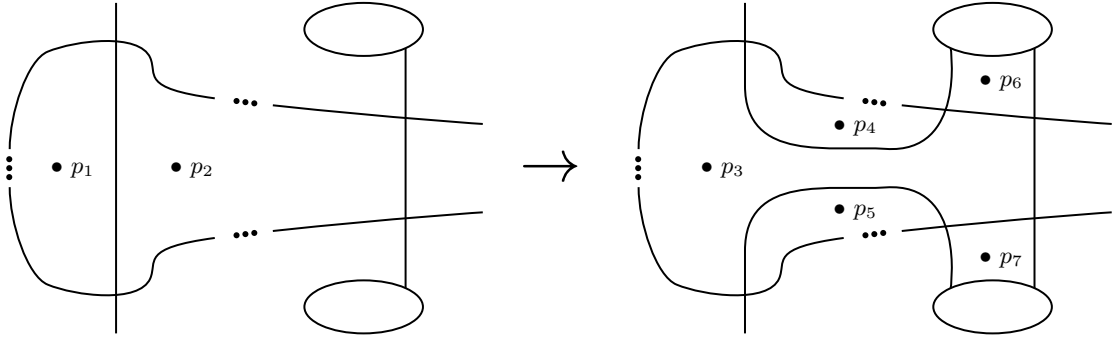As in the proof of Lemma \ref{isotopy}, $E_i$ is understood to be the grading equation for the disk in the complement of the vanishing cycles containing $p_i$, for some fixed labeling. To prove the first fact, in Figure \ref{slide} there are two arcs crossing the vertical vanishing cycle on the right. These could come from an arbitrary finger move, or, since this is coming from a crown diagram, they are the crossings coming from the vanishing cycles which are listed in $\Gamma$ before and after that vanishing cycle. By additivity of the grading equations, one may as well assume the regions containing $p_1,\ldots,p_7$ are as shown, possibly with corners hidden in the dotted parts of the segments, and no other vanishing cycles passing though their interiors. With this understood, the regions containing $p_6$ and $p_7$ can be interpreted as two pieces of a single rectangle whose grading equation will be called $E_6$. Then $E_1=E_3+E_6$, $E_2=E_4+E_5-E_6$, and the grading equations for all other regions before the slide are either unchanged or can be obtained by adding the grading equation for an appropriate rectangle in the annulus containing $p_6$ and $p_7$. Since the inverse of any slide is another slide up to isotopy, the first fact is proved.

For the second fact, \cite{W3} presents the sequence of slides for both multislide and handleslide as follows. Given a suitable pair of vanishing cycles for the move, choose a labeling so that they are $\gamma_1$ and $\gamma_n$, and such that the set of vanishing cycles affected by the move (that is, the nonstationary set) is $\{\gamma_2,\ldots,\gamma_{n-1}\}$. Then the move is achieved by applying a diffeomorphism of $\Sigma$ to the nonstationary set whose effect is a sequence of slides applied to the nonstationary set. This fact alone is enough to show that all of the salient crossings of the crown diagram except possibly $x_1$ and $x_{n-1}$ are preserved under handleslide and multislide. As for $x_1$ and $x_{n-1}$, the effect of the move near these crossings is depicted in Figure 8b of that paper, where the short vertical arc, contained in $\gamma_2$ or in $\gamma_{n-1}$, crosses the lower circle ($\gamma_1$ or $\gamma_n$ respectively) at the salient crossing $x_1$ or $x_{n-1}$, and that crossing is unchanged by the move.
\end{proof}
\begin{lemma}\label{shift}If two crown diagrams $(\Sigma_i,\Gamma_i)_{i=1,2}$ of $M$ are related by a shift move, then $\s(\Sigma_1,\Gamma_1)=\s(\Sigma_2,\Gamma_2)$.\end{lemma}
\begin{proof}
According to \cite[Theorem 3.1]{W3}, a shift move can be realized by a sequence of slides. As in the proof of Lemma \ref{slides}, any slide which does not change the collection of salient crossings will preserve $\s$; however, the argument in the proof of Lemma \ref{slides} that the collection of salient crossings is unchanged does not apply to the shift move, so there is one further argument to make. 

Applying the sequence of slides that produces a shift move can introduce new crossings between a consecutive pair of vanishing cycles $\gamma_i$ and $\gamma_{i+1}$. One of those new crossings $x'$ may be the only one left after canceling all the other crossings between $\gamma_i$ and $\gamma_{i+1}$ by Reidemeister-II$^{-1}$ moves in the course of an isotopy to obtain an allowable crown diagram. In this situation the original salient crossing $x$ can be cancelled last, and since $x'$ was not cancelled, it was one of the three corners in the last remaining pair of bigons, and 
in that situation it has already been shown that $\gr(x)=\gr(x')$. In this way, changing the salient crossing from $x$ to $x'$ does not change $\s$.\end{proof}
\begin{proof}[Proof of Theorem \ref{invthm}]According to \cite[Theorem 1.1]{W2}, any pair of crown diagrams in the same homotopy class become diffeomorphic after applying an isotopy and a sequence of moves chosen from a list of four moves called stabilization, handleslide, multislide, and shift. According to \cite[Theorem 4.9]{W3}, the stabilization move can be omitted from this list because the diagrams have the same genus. According to Lemmas \ref{isotopy}, \ref{slides}, and \ref{shift}, $\s$ is unchanged by isotopy and the remaining three moves.\end{proof}
\section{Examples}\label{examples}
\subsection{Relating crown diagrams of knot surgery manifolds}\ 
According to \cite[Theorem 1.9]{FS}, if $K$ is a fibered knot, then $E(1)$ is homeomorphic to the knot surgery manifold $E(1)_K$ and the Seiberg-Witten invariant of $E(1)_K$ is the Alexander polynomial $\Delta_K(t)$. In the same paper, Fintushel and Stern described a Lefschetz fibration structure on $E(1)_K$. The current paper uses a monodromy factorization of this Lefschetz fibration due to Yun \cite{Y}, and it refers to \href{https://knotinfo.math.indiana.edu/}{KnotInfo} for the Alexander polynomials and monodromy for the knots $7_6$ and $10_{133}$. These are genus-2 fibered knots with the same Alexander polynomial. Genus-1 knots would have simpler diagrams, but they are not used because the two fibered genus-1 knots are the trefoil $3_1$ and the figure-8 knot $4_1$, which are distinguished by their Alexander polynomials. 

It has been known since at least the publication of \cite{L} that one may convert a Lefschetz fibration over $S^2$ into a crown map by a homotopy of the fibration in which one converts the Lefschetz critical points into 3-cusped circles and performs cusp merges to unite the circles, resulting in a crown map whose crown diagram is called a \emph{coronation} of the Lefschetz fibration (see \cite[Figure 7]{W4} for an example of how this looks in a crown diagram). Using the notion of \emph{surgered diagrams} from \cite{W2} to deduce properties of the resulting vanishing cycles, Behrens and Hayano gave a description of the full collection of coronations of a Lefschetz fibration in terms of mapping class groups of surfaces \cite{BH}. In \cite{W4} there is a fully diagrammatic interpretation of the same homotopy, yielding a single, much more explicit, member of this collection. Recording the vanishing cycles using the notation of train tracks, the weights of such a diagram reach the tens of thousands for a diagram with 22 vanishing cycles, such as the ones for knot surgery on $E(1)$. Instead, this paper will consider a simpler kind of diagram which is slide-equivalent to any coronation of Yun's Lefschetz fibration. It comes with a collection of crossings that end up being the salient crossings of a coronation. Before discussing this, though, recall that the invariance of $\s$ applies to crown diagrams with a common homotopy class. The manifolds $\7$ and $\1$ are nominally different manifolds, so it is necessary to explain how their coronations would have a common homotopy class if their corresponding smoothings were isotopic.

\begin{lemma}\label{homotopylemma}Suppose $K$ and $K'$ are fibered knots in $S^3$, and let $f\co\x\to S^2$, $f'\co\y\to S^2$ be the Lefschetz fibrations studied in \cite{Y}.
\begin{enumerate}
\item There is a topological manifold $M$ and homeomorphisms $h\co M\to \x$, $h'\co M\to\y$ such that $f\circ h$ and $f'\circ h'$ are homotopic maps $M\to S^2$.
\item If the two smoothings coming from these Lefschetz fibrations are isotopic, then the salient sets of their coronations are equal.
\end{enumerate}\end{lemma}
\begin{proof}Let $M$ be a topological 4-manifold homeomorphic to $\CP\#9\CPb$ (and thus homeomorphic to the total spaces of $\x$ and $\y$). From the Pontrjagin-Thom construction, it is known that the collection of homotopy classes of maps $M\to S^2$ correspond to framed cobordism classes of framed surfaces in $M$, where the framed surface is obtained as the preimage of a neighborhood of a point in $S^2$. The goal is to find homeomorphisms $h,h'$ from $M$ to $\x$ and to $\y$ sending a cobordism class of surfaces in $M$ to the cobordism class of the fiber in $\x$ and in $\y$, and then to deal with the framing separately. Let $h\co M\to\x$ and $\phi\co M\to\y$ be homeomorphisms, and let $c\in H^2(\x)$ and $c'\in H_2(\y)$ be the fiber classes of $f$ and $f'$, respectively. Then neither $(\phi\circ h^{-1})_*(c)$ nor $c'$ are characteristic homology classes in $\y$, and both are primitive, because $f$ and $f'$ each have a section whose self-intersection number is $-2$ \cite[Section 6]{FS}, while its intersection with the fiber class is 1. Thus, by \cite[Proposition 1.2.18]{GS}, there is an automorphism  $\alpha$ of the intersection pairing of $\y$ sending $(\phi\circ h^{-1})_*(c)$ to $c'$. According to the topological $h$-cobordism theorem, $\alpha$ is realized as the induced map of a self-homeomorphism $\psi$ of $\y$ (see for example \cite[Chapter 9]{GS}). Thus, $h'=\psi\circ\phi$ is a homeomorphism such that $h^{-1}_*(c)=(h')^{-1}_*(c')$, which yields the required correspondence of cobordism classes.\[\begin{tikzcd}[column sep=tiny]
&\ar[dl, "h"']M \ar[dr, "h'=\psi\circ\phi"] &\\
\x\ar[dr, "f"]&&\ar[dl, "f'"']\y\\
&S^2&
\end{tikzcd}\]

Having found $h,h'$ such that $(f\circ h)^{-1}(pt)$ is cobordant in $M$ to $(f'\circ h')^{-1}(pt)$, let $\Sigma\subset M$ be a surface in that cobordism class. According to Theorem 2 and Example 1.3 of \cite{KMT}, there is a unique framing of $\Sigma$ up to framed cobordism because $M$ is simply connected and its intersection form is odd. For this reason, $h$ and $h'$ send the unique framed cobordism class of $\Sigma$ to that of the fiber of $f$ and $f'$, respectively. For this reason, $f_0=f\circ h$ and $f_1=f'\circ h'$ are continuously homotopic maps $M\to S^2$. This completes the proof of Item (1).

For Item (2), since the smoothings coming from $f_0$ and $f_1$ are isotopic, there is an isotopy of $f_1$ converting it to be a smooth map on the total space of $f_0$. Thus, $f_0$ and $f_1$ can be taken as smooth maps on the ends of $\x\times[0,1]$ connected by a continuous homotopy. That is, they share a homotopy class.

To the ends of this homotopy, append the further homotopies which convert the Lefschetz fibrations $f_0$ and $f_1$ into crown maps $c_0,c_1$ using \cite[Section 4]{W1}. Then, according to \cite[Section 4]{W3}, these can be connected by a deformation corresponding to a sequence of genus-preserving moves on crown maps, which have equal salient sets by Theorem \ref{invthm}.\end{proof}
\begin{q}It seems natural to ask whether two diffeomorphic smoothings realized as the total spaces of homotopic fibration maps is enough to imply that the two smoothings are isotopic. In that case, the main result of the paper could be immediately strengthened to state that the manifolds are not diffeomorphic.\end{q}
\subsection{Diagrams for two knot surgery manifolds}
Figure \ref{yunfig} depicts the relevant information to obtain Yun's monodromy factorizations for $\7$ and $\1$. Starting with the five circles in Figure \ref{yunfig1}, cut the surface at the four circles shown in Figure \ref{yunfig2}. Let $t_x$ denote the right Dehn twist about the circle $x$. Then, according to KnotInfo, and using the circles shown in Figure \ref{yunfig4}, the monodromy of $7_6$ can be written as the composition of Dehn twists \[\phi_{7_6}=t_a\cdot t_b\cdot t_c^{-1}\cdot t_d,\] applied from left to right, while the monodromy of $10_{133}$ can be written \[\phi_{10_{133}}=t_a\cdot t_b\cdot t_c\cdot t_d^{-1}\cdot t_d^{-1}\cdot t_f.\] According to \cite[Theorem 2.11]{Y}, the Lefschetz monodromy factorization for $E(1)_K$ is \[\left(\prod\limits_{i=0}^4t_{\phi_K(B_i)}\right)^2\cdot\left(\prod\limits_{i=0}^4t_{B_i}\right)^2.\]
The collection of circles that results for $7_6$ is a messy picture, which can be simplified considerably by applying one Dehn twist and pushing the feet of two of the handles around, switching their places in a clockwise rotation as in Figure \ref{diffeo}. The resulting collection of Lefschetz vanishing cycles can be seen in Figures \ref{pc-7-6-1-6}--\ref{pc-7-6-15-21} by deleting the pair of long rectangles and connecting the resulting ends with vertical line segments. The simplifying diffeomorphism is not used for $\1$, and the resulting collection of circles appears analogously in Figure \ref{pc-10-133}. Note that in those figures, the reversal of ordering of the circles comes from the fact that the circles in the above monodromy factorization are vanishing cycles of a Lefschetz fibration, while the circles in Figures \ref{pc-7-6} and \ref{pc-10-133} get their ordering from the merging algorithm depicted in \cite[Figure 6]{W4}. In short, the so-called \emph{pseudocoronation} shown in Figure \ref{pc-7-6} or Figure \ref{pc-10-133} comes from choosing an arc in the fiber of the Lefschetz fibration that crosses each of the fibration's Lefschetz vanishing cycles exactly once, replacing that arc with a tube over which all of those vanishing cycles run, adding two more vanishing cycles in a standard way as in Figures \ref{pc-7-6-20-22} and \ref{pc-10-133-20-22}, then ordering the vanishing cycles as required by \cite[Figure 6]{W4}: In that figure, $r_i$ or $r_i^s$ corresponds to the $i^{th}$ Lefschetz vanishing cycle. Circle 20 in 
Figures in Figures \ref{pc-7-6-20-22} and \ref{pc-10-133-20-22} corresponds to $g_1$ in \cite[Figure 6]{W4}, and circle 22 in Figures \ref{pc-7-6-20-22} and \ref{pc-10-133-20-22} corresponds to $b_1$ in \cite[Figure 6]{W4}.

\begin{figure}
	\centering
	\begin{subfigure}{0.4\textwidth}
		\labellist
		\small\hair 2pt
		\pinlabel $B_0$ at 100 250
		\pinlabel $B_1$ at 135 230
		\pinlabel $B_2$ at 135 204
		\pinlabel $B_3$ at 135 184
		\pinlabel $B_4$ at 135 160
		\endlabellist
	\includegraphics[width=\textwidth]{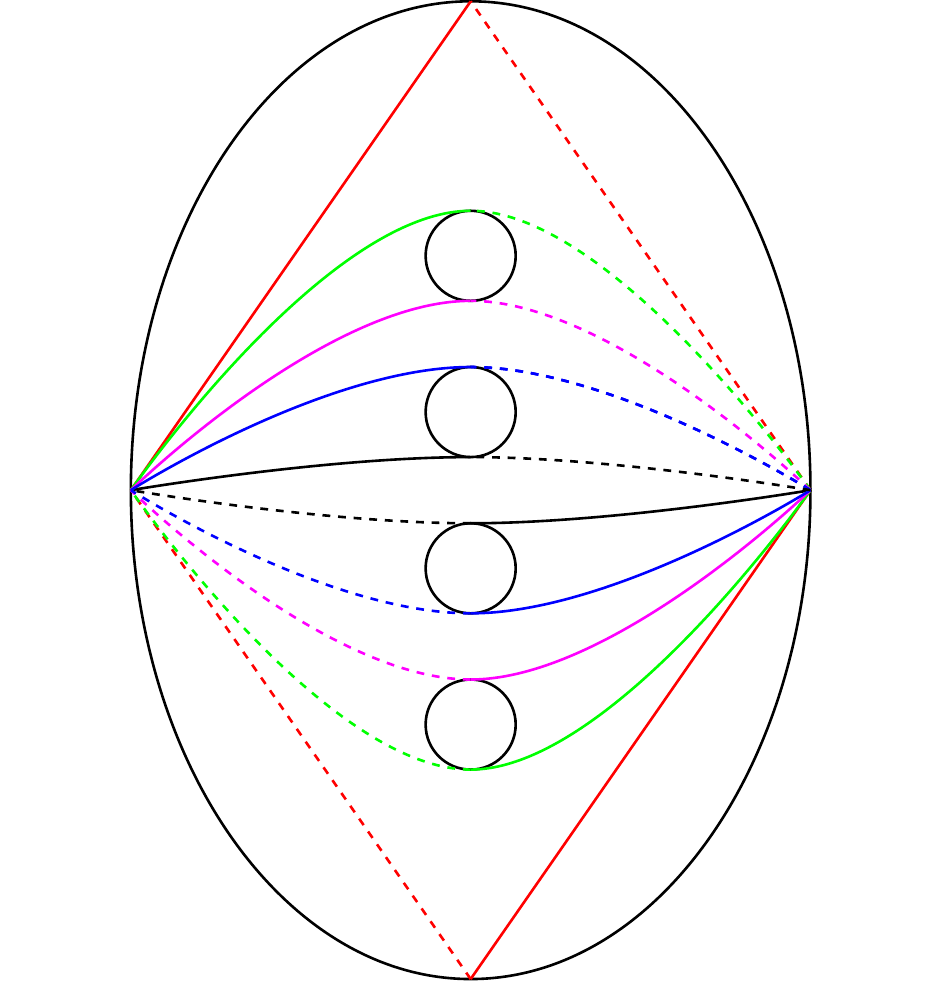}
	\caption{\cite[Figure 3]{Y} in the case where $g=2$ and $n=1$.}
	\label{yunfig1}
	\end{subfigure}\hspace{.05\textwidth}
	\begin{subfigure}{0.4\textwidth}
		\labellist
		\small\hair 2pt
		\pinlabel 1 at 135 250
		\pinlabel 2 at 135 186
		\pinlabel 3 at 135 97
		\pinlabel 4 at 135 30
		\endlabellist
	\includegraphics[width=\textwidth]{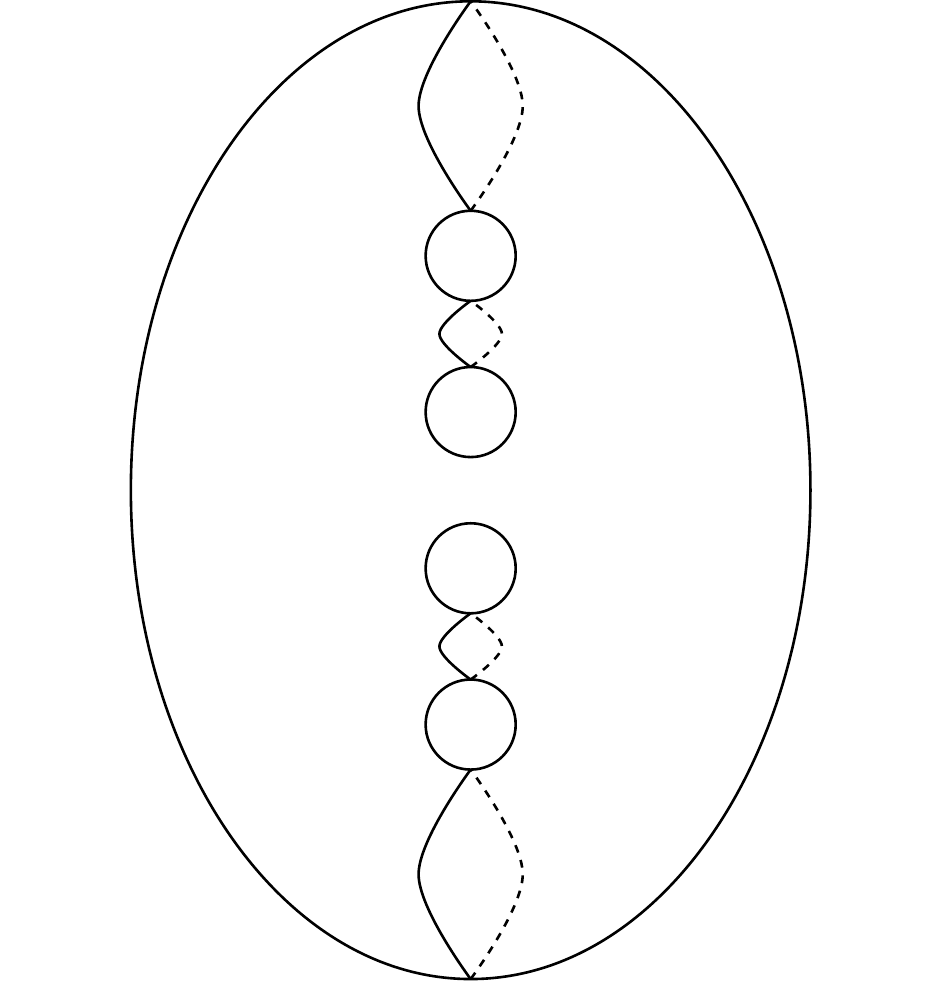}
	\caption{Cuts used to produce the diagrams in Figures \ref{yunfig3} and after.}
	\label{yunfig2}
	\end{subfigure}\\
\begin{subfigure}{0.4\textwidth}
		\labellist
		\small\hair 2pt
		\pinlabel $B_0$ at 60 222
		\pinlabel $B_1$ at 159 200
		\pinlabel $B_2$ at 33 192
		\pinlabel $B_3$ at 120 108
		\pinlabel $B_4$ at 152 24
		\pinlabel 1 at 95 215
		\pinlabel 1 at 175 215
		\pinlabel 2 at 95 165
		\pinlabel 2 at 175 165
		\pinlabel 3 at 95 115
		\pinlabel 3 at 175 115
		\pinlabel 4 at 95 65
		\pinlabel 4 at 175 65
		\endlabellist
	\includegraphics[width=\textwidth]{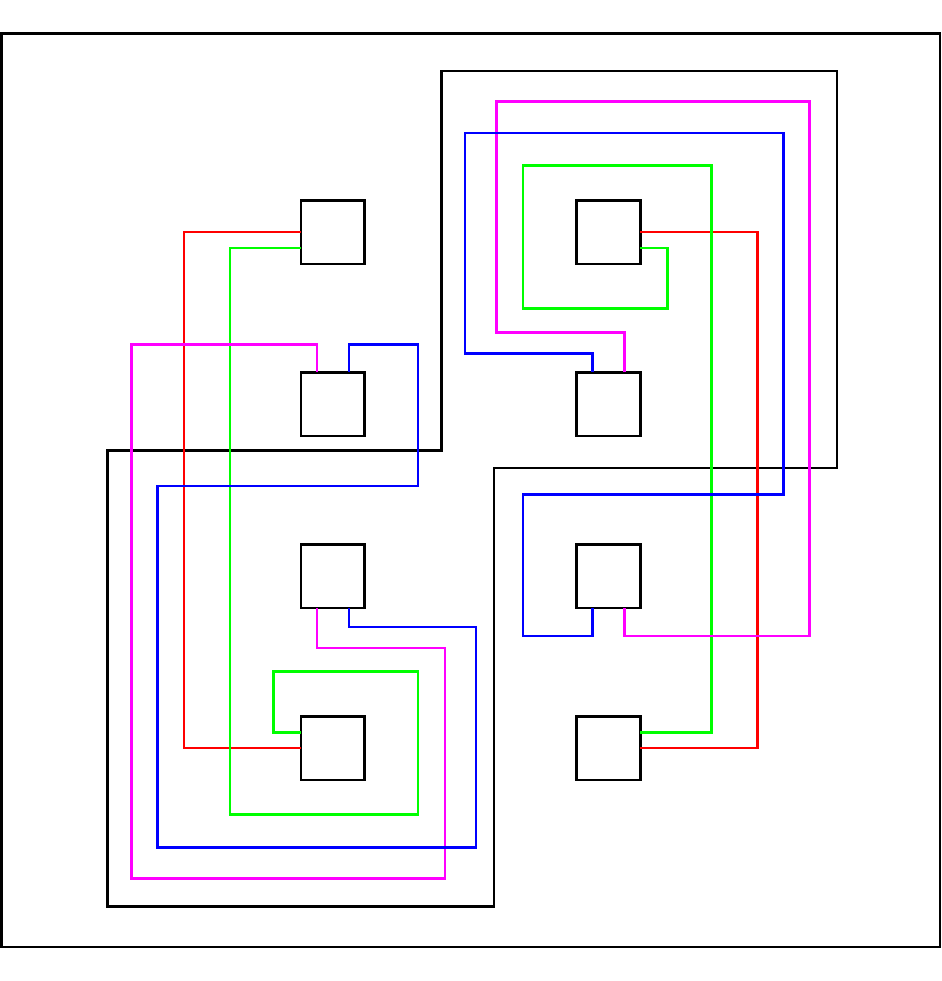}
	\caption{These are circles of Figure \ref{yunfig1} after cutting the surface.}
	\label{yunfig3}
	\end{subfigure}\hspace{.05\textwidth}
	\begin{subfigure}{0.4\textwidth}
		\labellist
		\small\hair 2pt
		\pinlabel $f$ at 70 230
		\pinlabel $a$ at 90 185
		\pinlabel $c$ at 120 172
		\pinlabel $d$ at 208 230
		\pinlabel $b$ at 191 191
		\endlabellist
	\includegraphics[width=\textwidth]{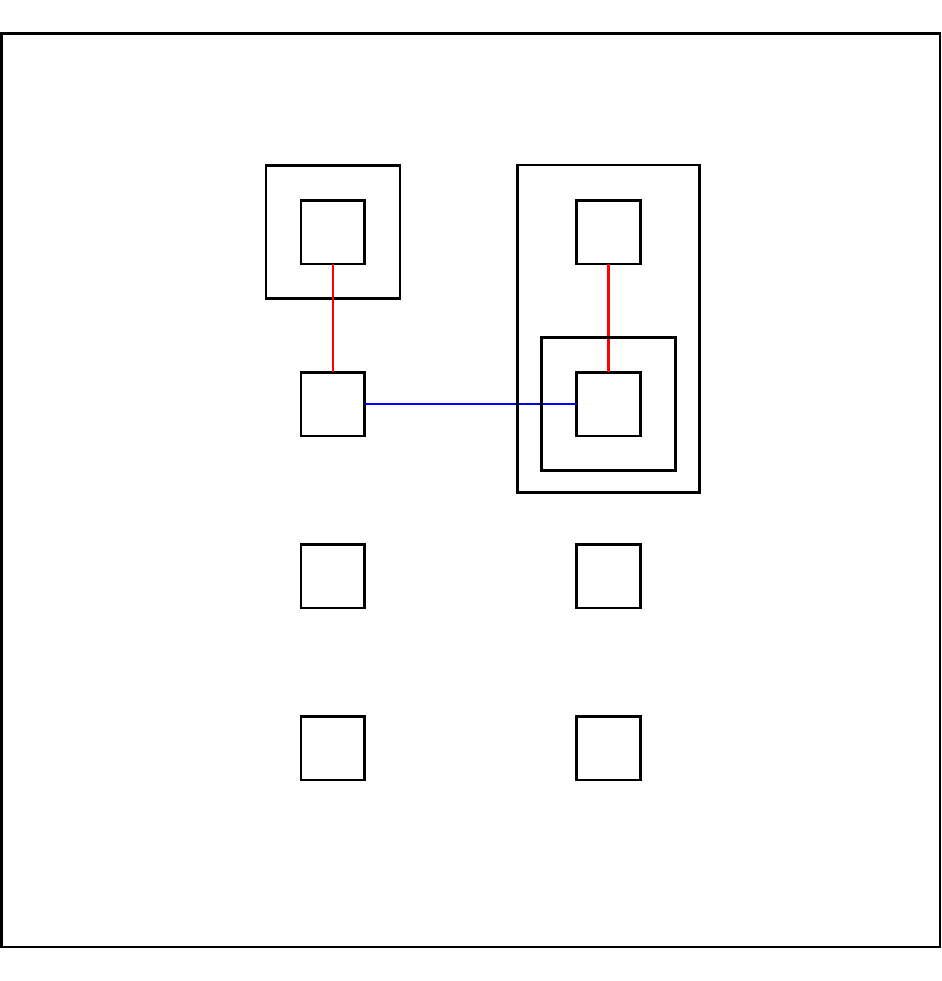}
	\caption{The circles $a$--$f$ in KnotInfo's factorization of the monodromy for $7_6$ and $10_{133}$.}
	\label{yunfig4}
	\end{subfigure}
	\caption{Figures for finding the Lefschetz vanishing cycles for $\7$ and $\1$.}\label{yunfig}
\end{figure}

According to the algorithm in \cite{W4} and the subsequent proofs that the algorithm is realizable as a sequence of slides between vanishing cycles in the pseudocoronation, a coronation of either of the Lefschetz fibrations of $\7$ or $\1$ can be obtained from applying slides among the vanishing cycles in Figures \ref{pc-7-6} and \ref{pc-10-133}, while keeping the staircase-shaped arrangement in Figures \ref{pc-7-6-all} and \ref{pc-10-133-all} unchanged. Since $\s$ is invariant under such slides, one could instead use the grading system as given by the depicted pseudocoronations, taking the crossings between consecutive vanishing cycles in the staircase arrangement as the salient crossings: The first step in the algorithm for finding the vanishing cycle $\gamma_i$ of the coronation is to repeatedly slide $\gamma_i$ over $\gamma_{22}$ along $\gamma_{i+1}$ to eliminate all crossings between $\gamma_i$ and $\gamma_{i+1}$ other than the salient crossing in the staircase.
 
\begin{figure}
	\centering
	\begin{subfigure}{0.3\textwidth}
		\labellist
		\small\hair 2pt
		\pinlabel 1 at 34 231
		\pinlabel 1 at 140 231
		\pinlabel 2 at 34 165
		\pinlabel 2 at 140 165
		\pinlabel 3 at 34 99
		\pinlabel 3 at 140 99
		\pinlabel 4 at 34 33
		\pinlabel 4 at 140 33
		\pinlabel 5 at 310 149
		\pinlabel 5 at 310 115
		\pinlabel $c$ at 180 200
		\pinlabel $\rightarrow$ at 430 132
		\endlabellist
		\includegraphics[width=\textwidth]{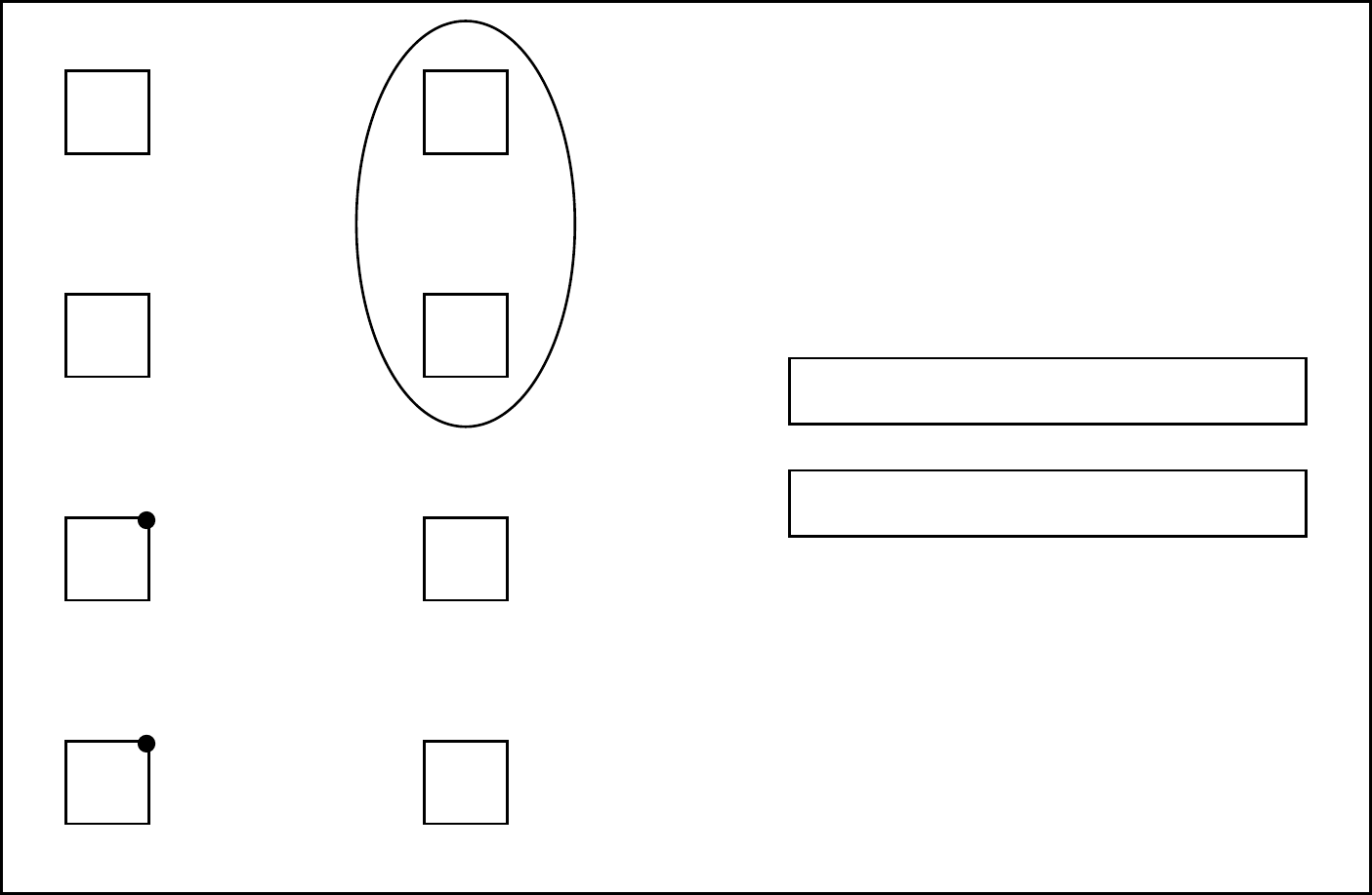}
		\caption{\ }
		\label{diffeo1}
	\end{subfigure}\hspace{.5cm}
	\begin{subfigure}{0.3\textwidth}
		\labellist
		\small\hair 2pt
		\pinlabel 1 at 34 231
		\pinlabel 1 at 140 231
		\pinlabel 2 at 34 165
		\pinlabel 2 at 140 165
		\pinlabel 3 at 75 84
		\pinlabel 3 at 140 99
		\pinlabel 4 at 34 33
		\pinlabel 4 at 140 33
		\pinlabel 5 at 310 149
		\pinlabel 5 at 310 115
		\pinlabel $\rightarrow$ at 430 132
\endlabellist
		\includegraphics[width=\textwidth]{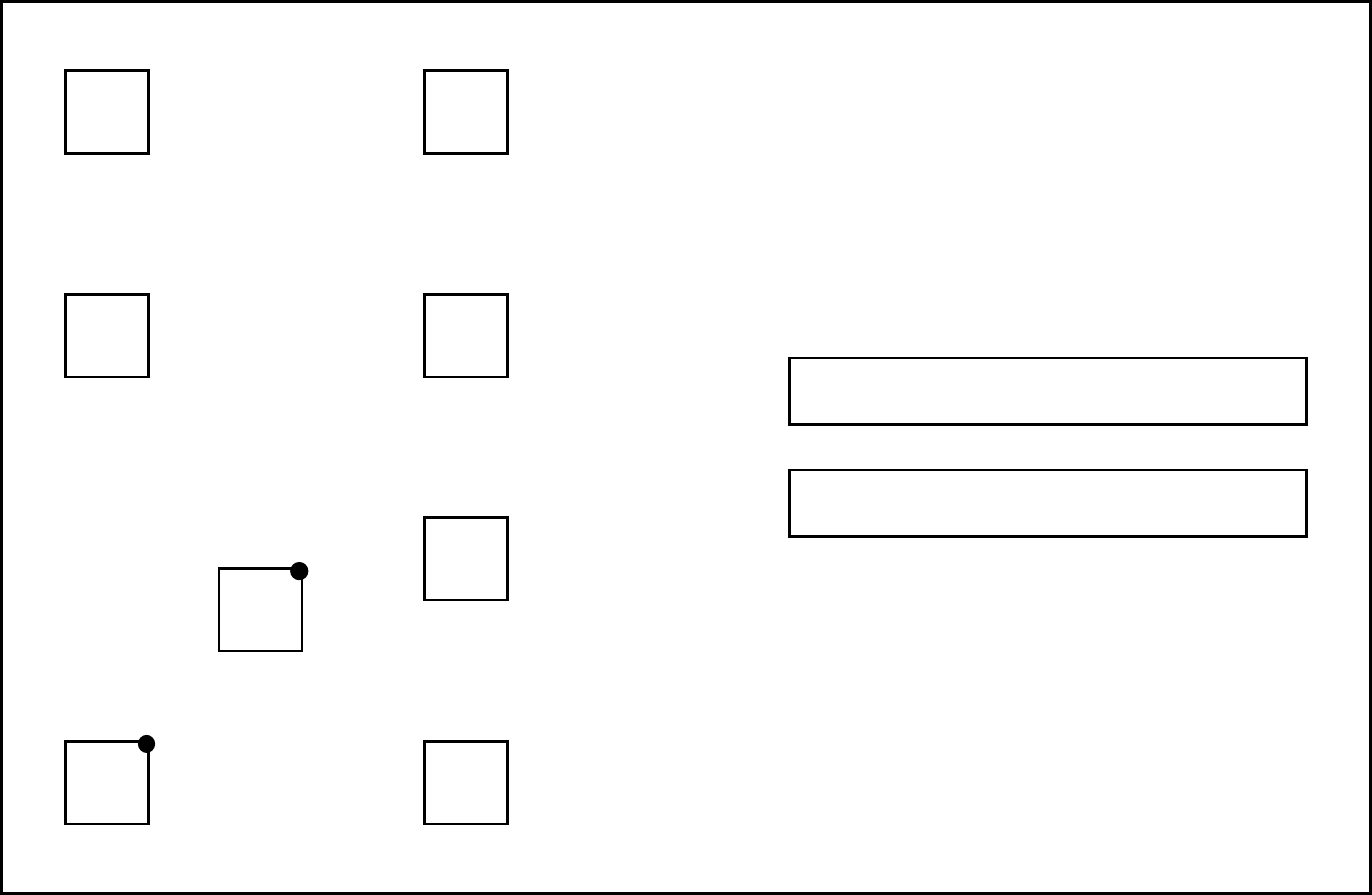}
		\caption{\ }
		\label{diffeo2}
	\end{subfigure}\hspace{.5cm}
	\begin{subfigure}{0.3\textwidth}
		\labellist
		\small\hair 2pt
		\pinlabel 1 at 34 231
		\pinlabel 1 at 140 231
		\pinlabel 2 at 34 165
		\pinlabel 2 at 140 165
		\pinlabel 3 at 77 55
		\pinlabel 3 at 140 99
		\pinlabel 4 at 33 78
		\pinlabel 4 at 140 33
		\pinlabel 5 at 310 149
		\pinlabel 5 at 310 115
		\pinlabel $\swarrow$ at -25 -25
		\endlabellist
		\includegraphics[width=\textwidth]{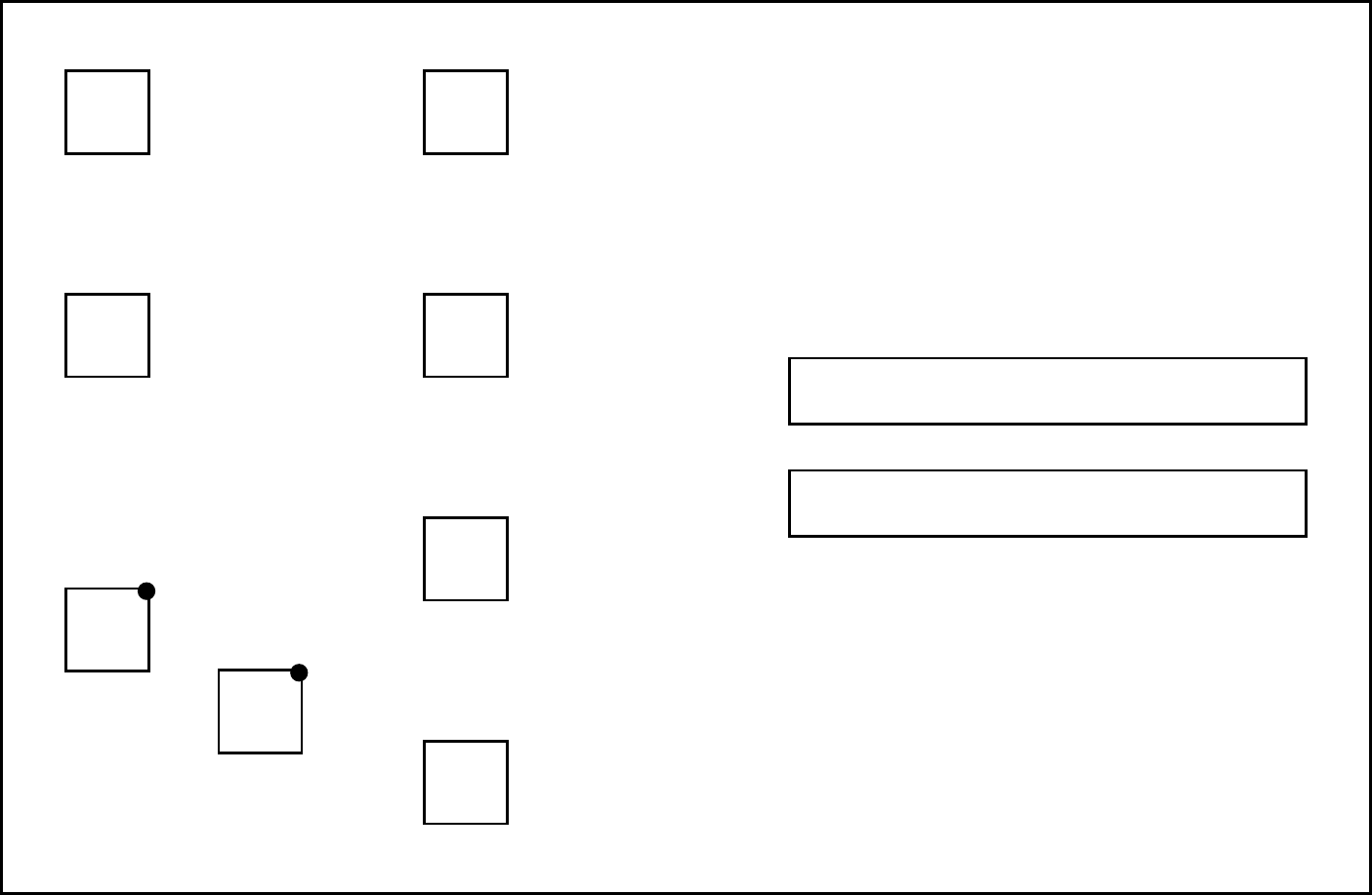}
		\caption{\ }
		\label{diffeo3}
	\end{subfigure}
	\begin{subfigure}{0.3\textwidth}
		\labellist
		\small\hair 2pt
		\pinlabel 1 at 34 231
		\pinlabel 1 at 140 231
		\pinlabel 2 at 34 165
		\pinlabel 2 at 140 165
		\pinlabel 4 at 34 99
		\pinlabel 3 at 140 99
		\pinlabel 3 at 34 33
		\pinlabel 4 at 140 33
		\pinlabel 5 at 310 149
		\pinlabel 5 at 310 115
		\endlabellist
		\includegraphics[width=\textwidth]{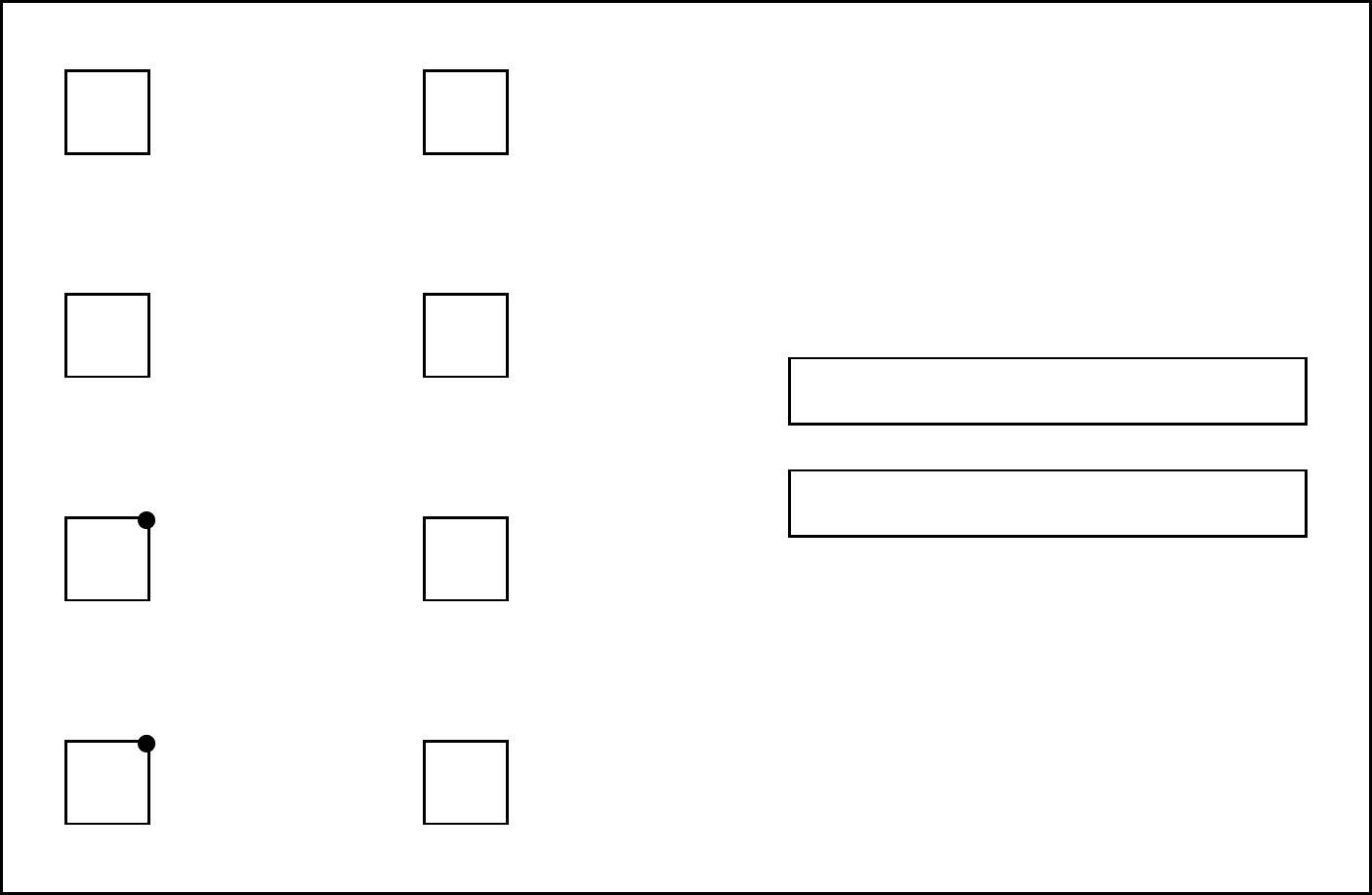}
		\caption{\ }
		\label{diffeo4}
	\end{subfigure}
\caption{Two diffeomorphisms which are used to simplify the depiction of the Lefschetz vanishing cycles for $\7$. The first is a negative Dehn twist about the circle $c$. The second moves the feet of the tubes 3 and 4 so that they switch places. The marked points are meant to illustrate how the squares do not rotate as they move.\label{diffeo}}
\end{figure}

\begin{figure}
\centering
\begin{subfigure}{0.23\textwidth}
\labellist
\small\hair 2pt
\pinlabel 6 at 300 195
\endlabellist
\includegraphics[width=\textwidth]{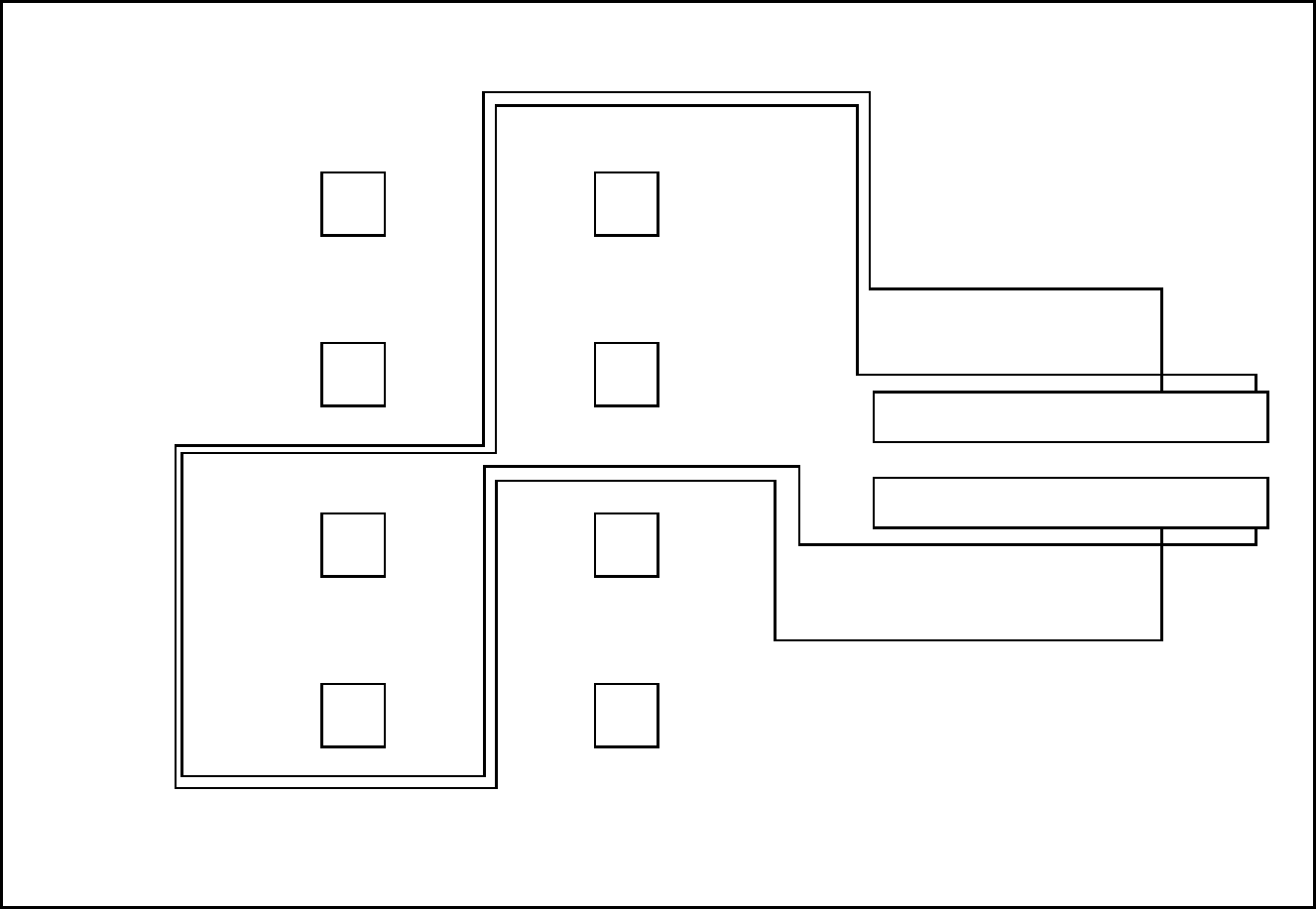}
\caption{Circles 1 and 6.}
\label{pc-7-6-1-6}
\end{subfigure}\hspace{.15cm}
\begin{subfigure}{0.23\textwidth}
\labellist
\small\hair 2pt
\pinlabel 7 at 270 200
\endlabellist
\includegraphics[width=\textwidth]{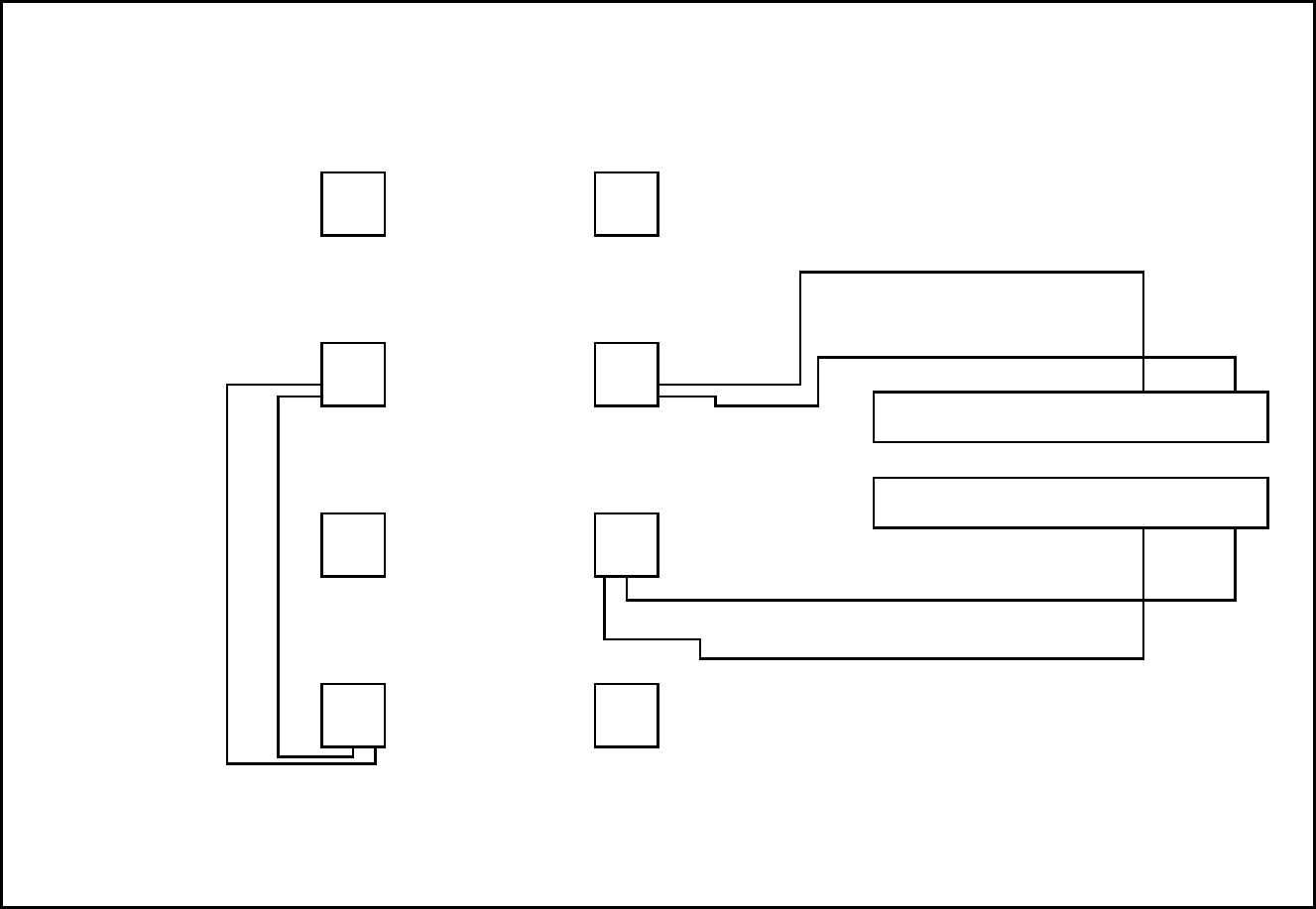}
\caption{Circles 2 and 7.}
\label{pc-7-6-2-7}
\end{subfigure}\hspace{.15cm}
\begin{subfigure}{0.23\textwidth}
\labellist
\small\hair 2pt
\pinlabel 8 at 270 205
\endlabellist
\includegraphics[width=\textwidth]{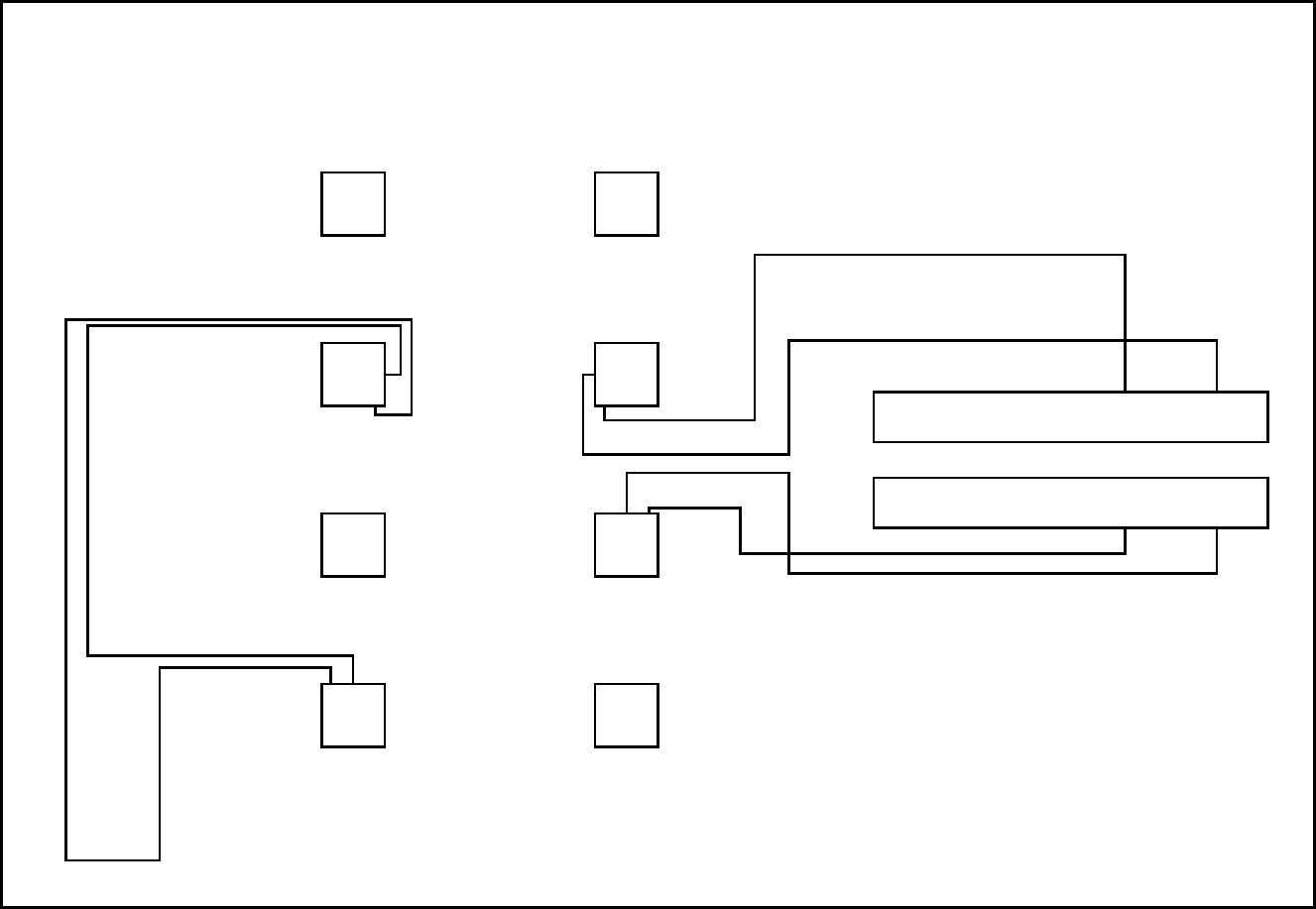}
\caption{Circles 3 and 8.}
\label{pc-7-6-3-8}
\end{subfigure}\hspace{.15cm}
\begin{subfigure}{0.23\textwidth}
\labellist
\small\hair 2pt
\pinlabel 9 at 270 210
\endlabellist
\includegraphics[width=\textwidth]{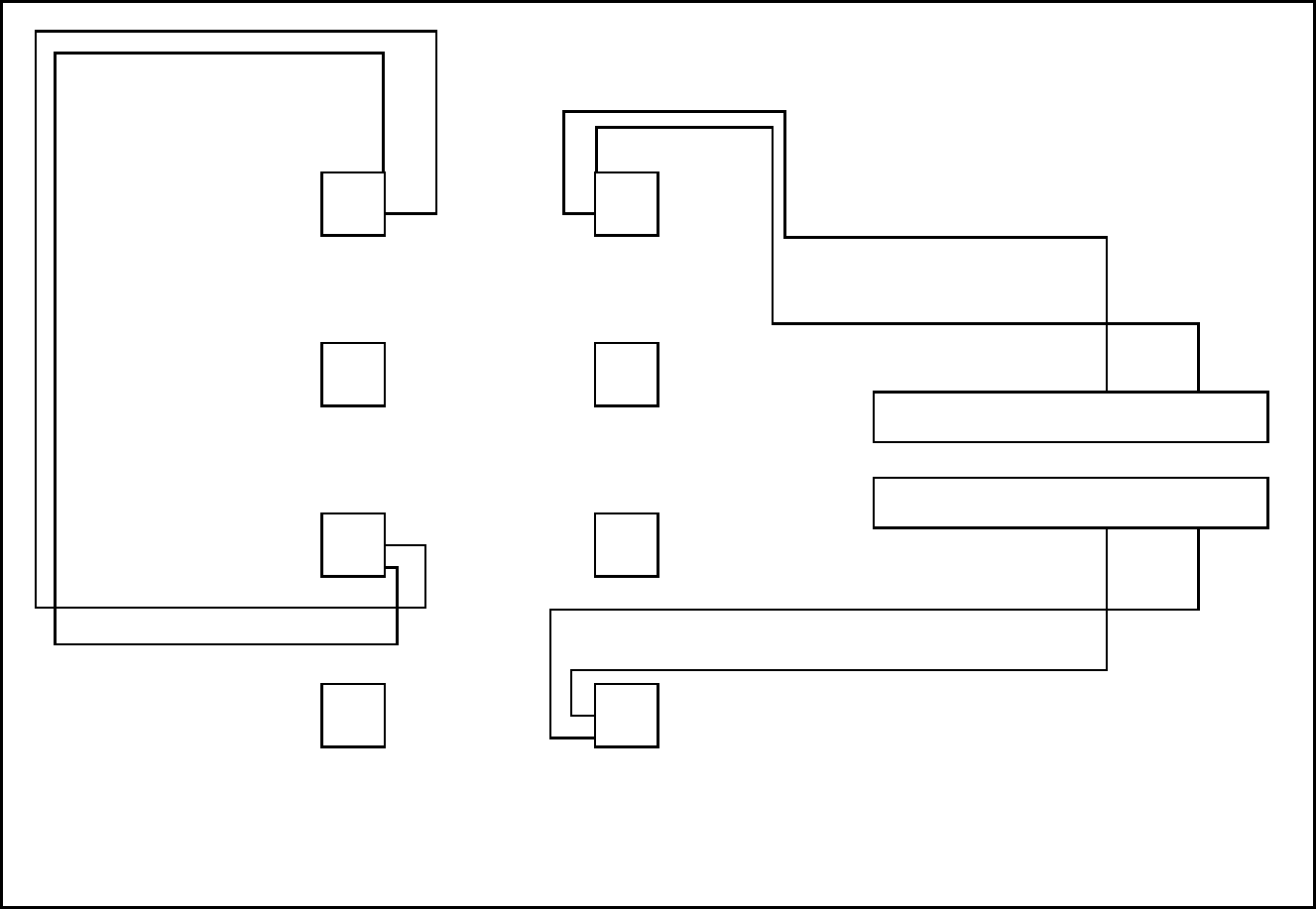}
\caption{Circles 4 and 9.}
\label{pc-7-6-4-9}
\end{subfigure}\vspace{.25cm}
\begin{subfigure}{0.23\textwidth}
\labellist
\small\hair 2pt
\pinlabel 10 at 270 215
\endlabellist
\includegraphics[width=\textwidth]{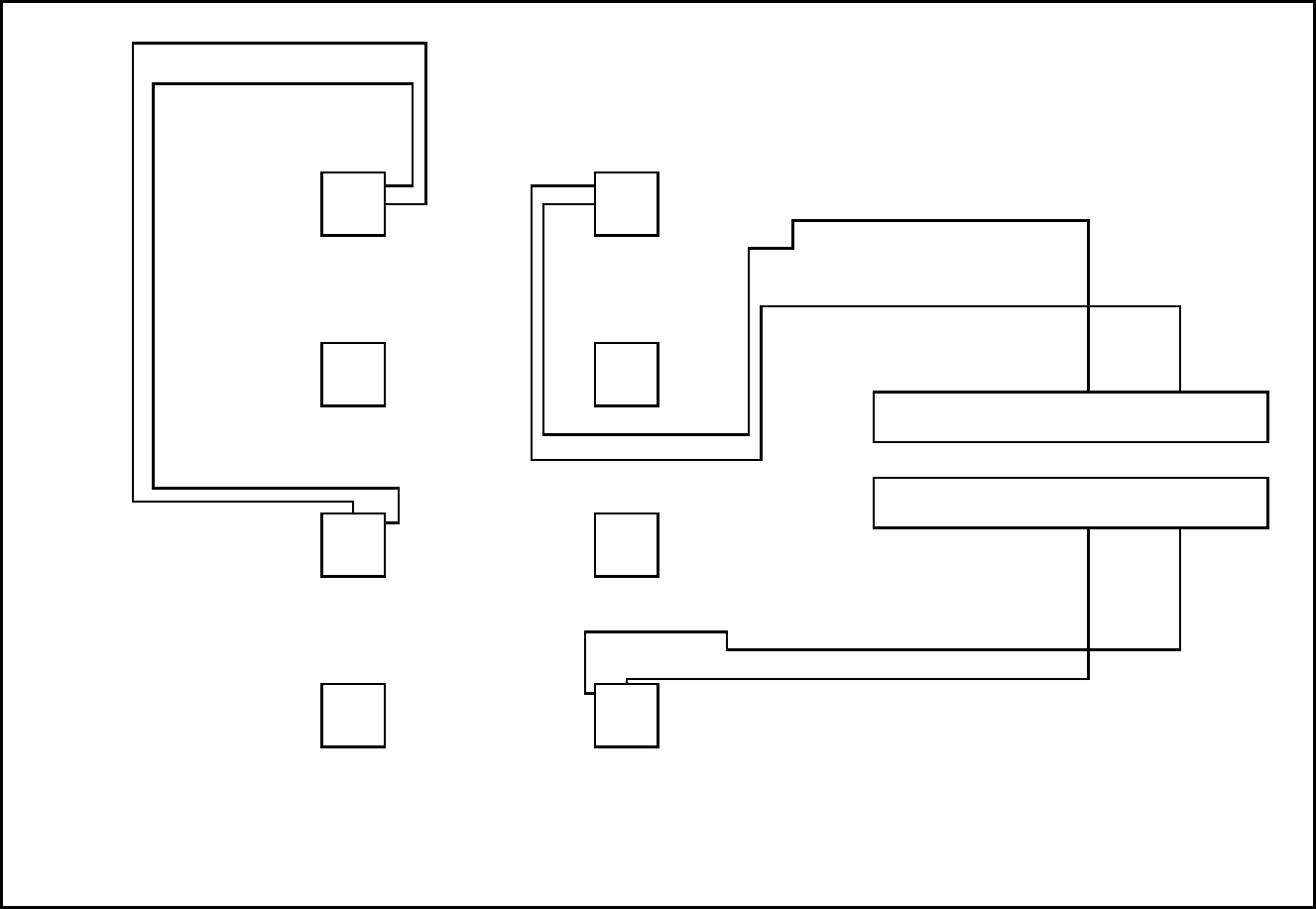}
\caption{Circles 5 and 10.}
\label{pc-7-6-5-10}
\end{subfigure}\hspace{.15cm}
\begin{subfigure}{0.23\textwidth}
\labellist
\small\hair 2pt
\pinlabel 11 at 330 175
\endlabellist
\includegraphics[width=\textwidth]{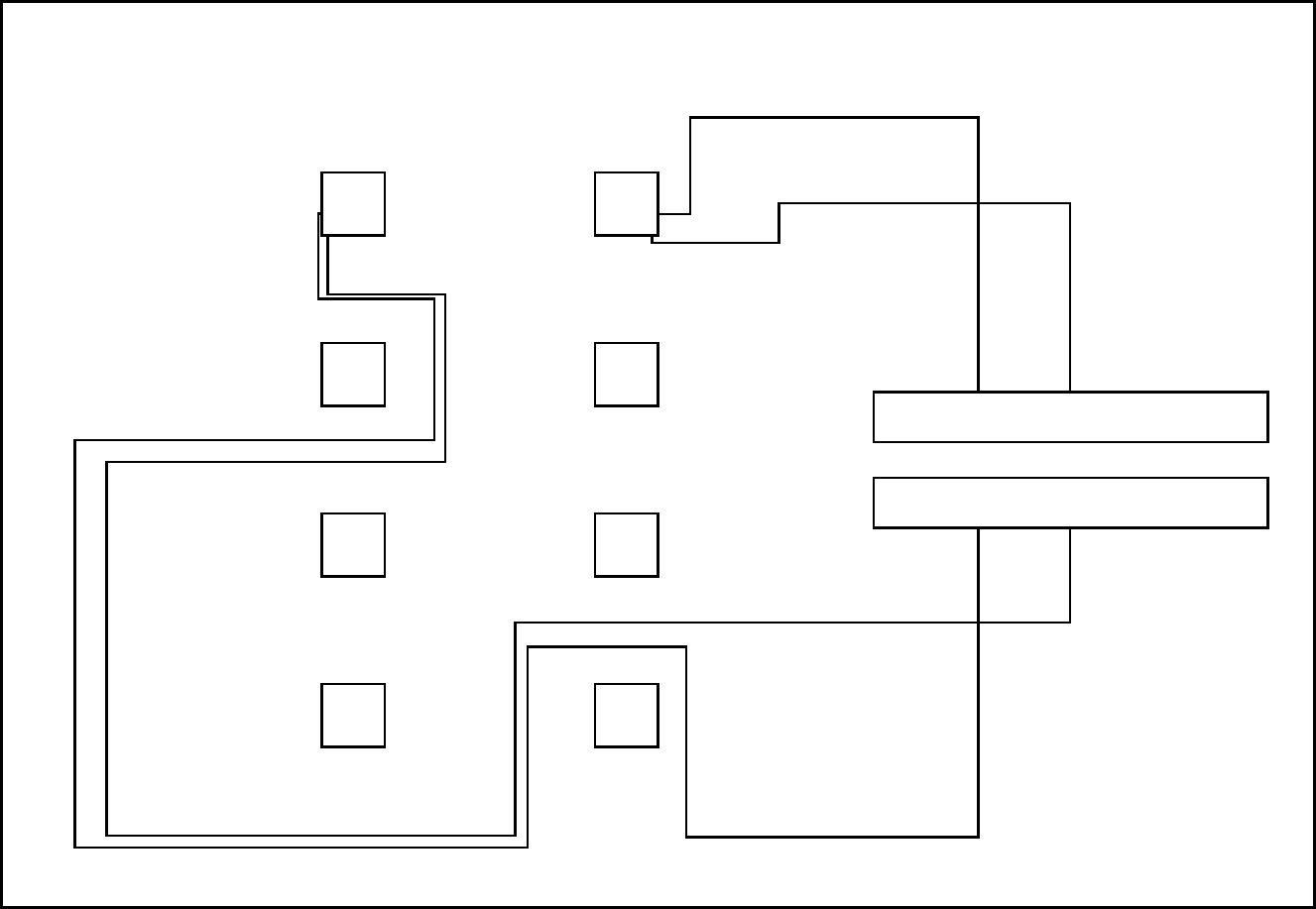}
\caption{Circles 11 and 16.}
\label{pc-7-6-11-16}
\end{subfigure}\hspace{.15cm}
\begin{subfigure}{0.23\textwidth}
\labellist
\small\hair 2pt
\pinlabel 12 at 325 175
\endlabellist
\includegraphics[width=\textwidth]{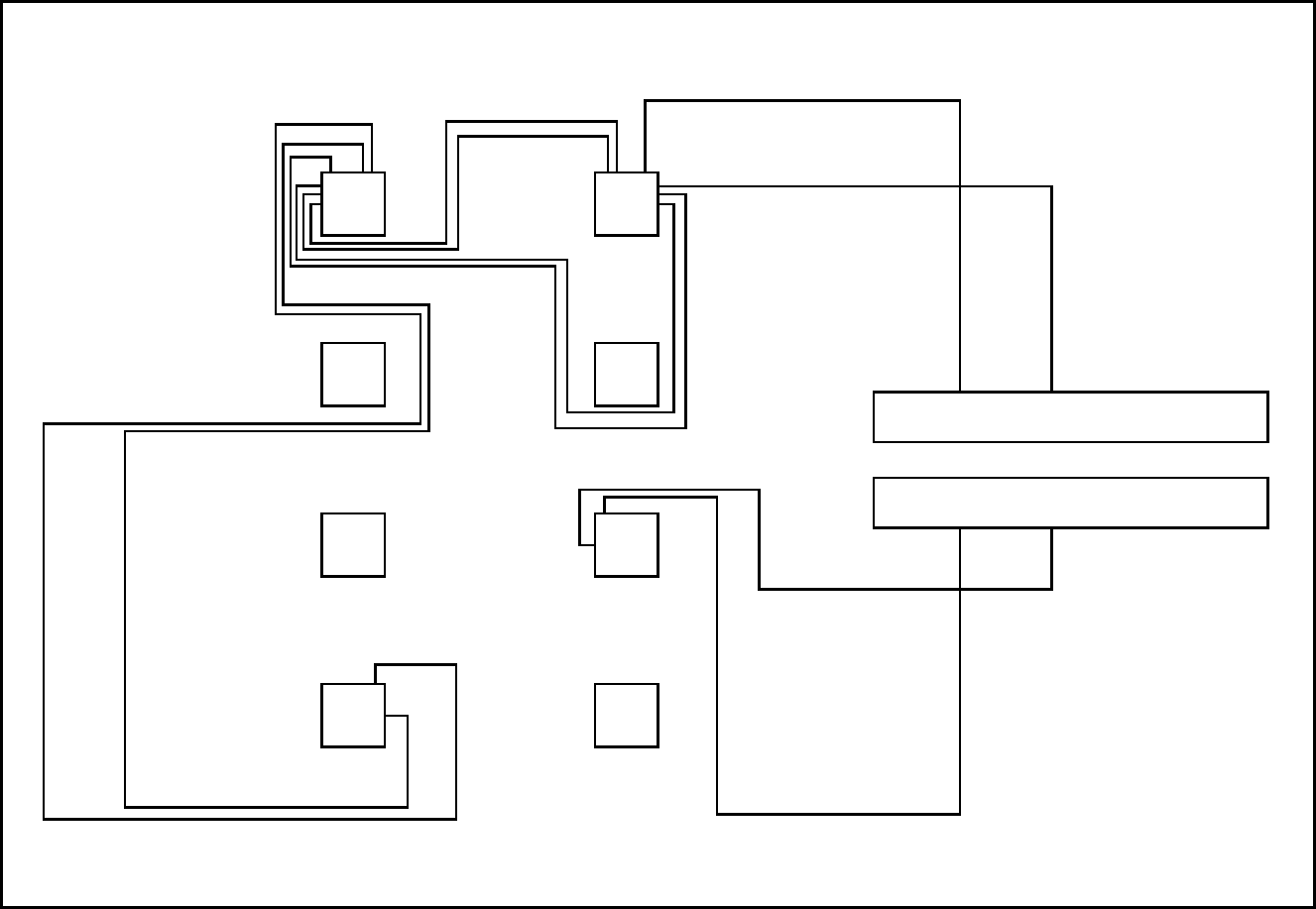}
\caption{Circles 12 and 17.}
\label{pc-7-6-12-17}
\end{subfigure}\hspace{.15cm}
\begin{subfigure}{0.23\textwidth}
\labellist
\small\hair 2pt
\pinlabel 13 at 320 175
\endlabellist
\includegraphics[width=\textwidth]{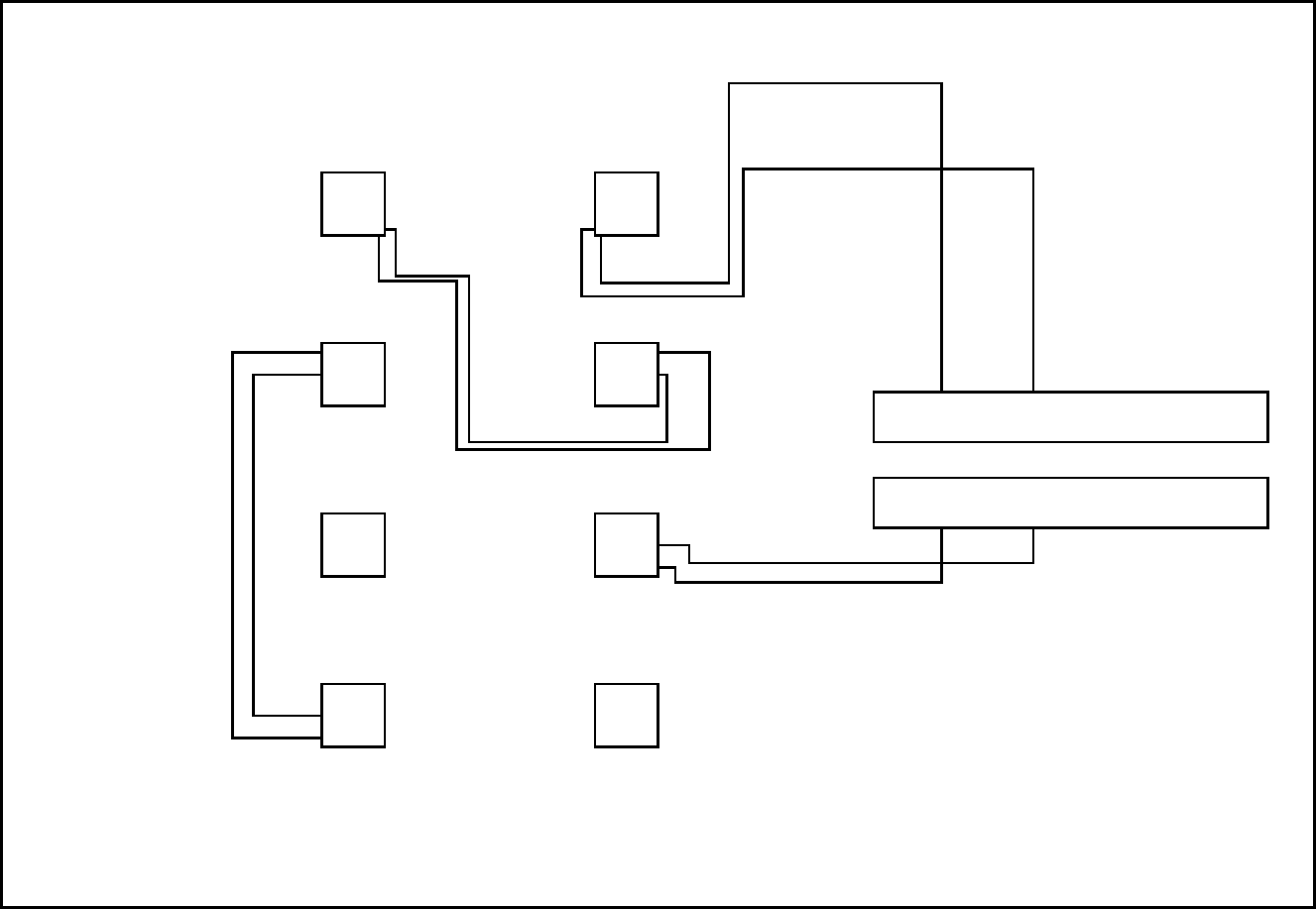}
\caption{Circles 13 and 18.}
\label{pc-7-6-13-18}
\end{subfigure}\vspace{.25cm}
\begin{subfigure}{0.23\textwidth}
\labellist
\small\hair 2pt
\pinlabel 14 at 315 175
\endlabellist
\includegraphics[width=\textwidth]{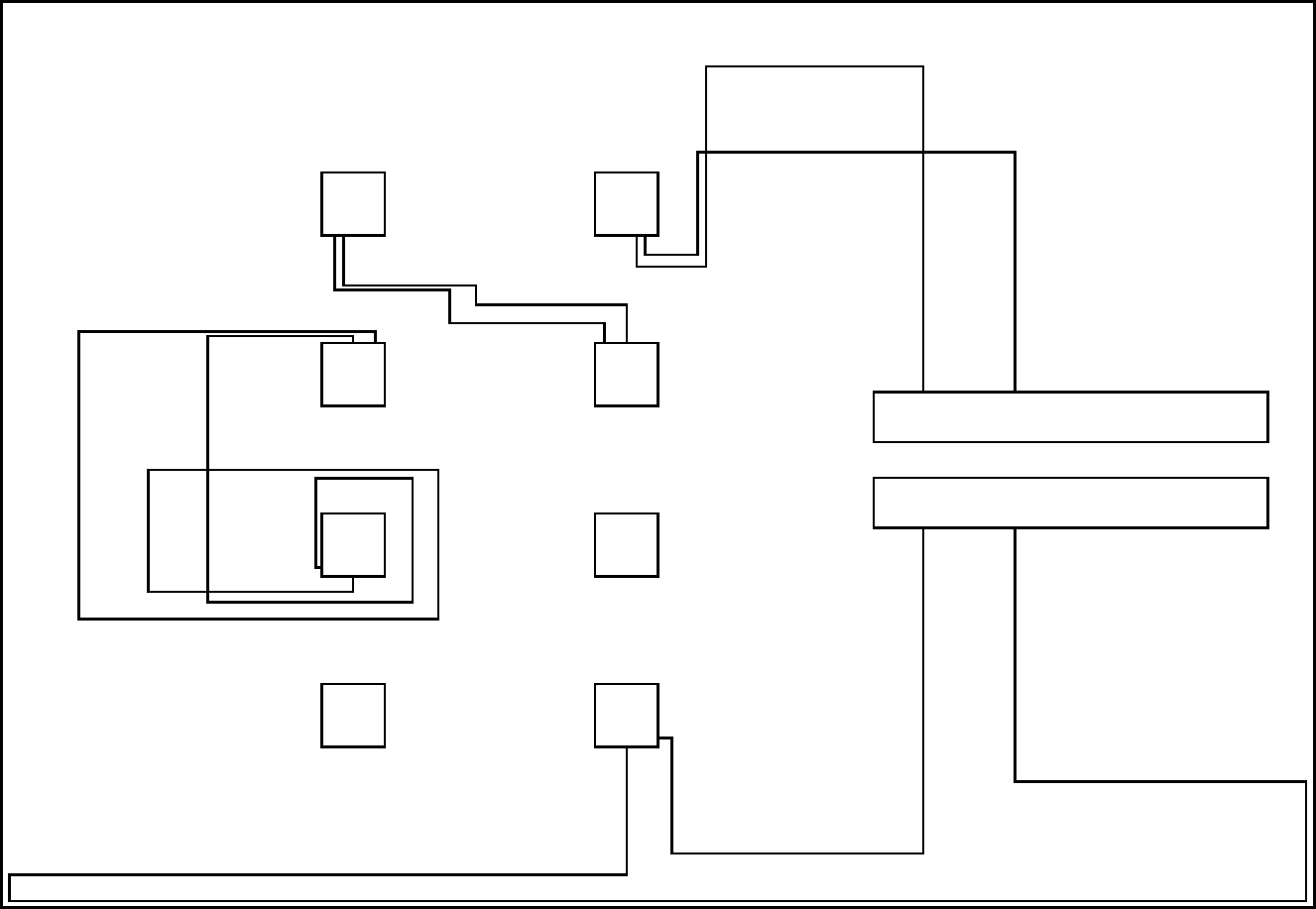}
\caption{Circles 14 and 19.}
\label{pc-7-6-14-19}
\end{subfigure}\hspace{.15cm}
\begin{subfigure}{0.23\textwidth}
\labellist
\small\hair 2pt
\pinlabel 15 at 310 175
\endlabellist
\includegraphics[width=\textwidth]{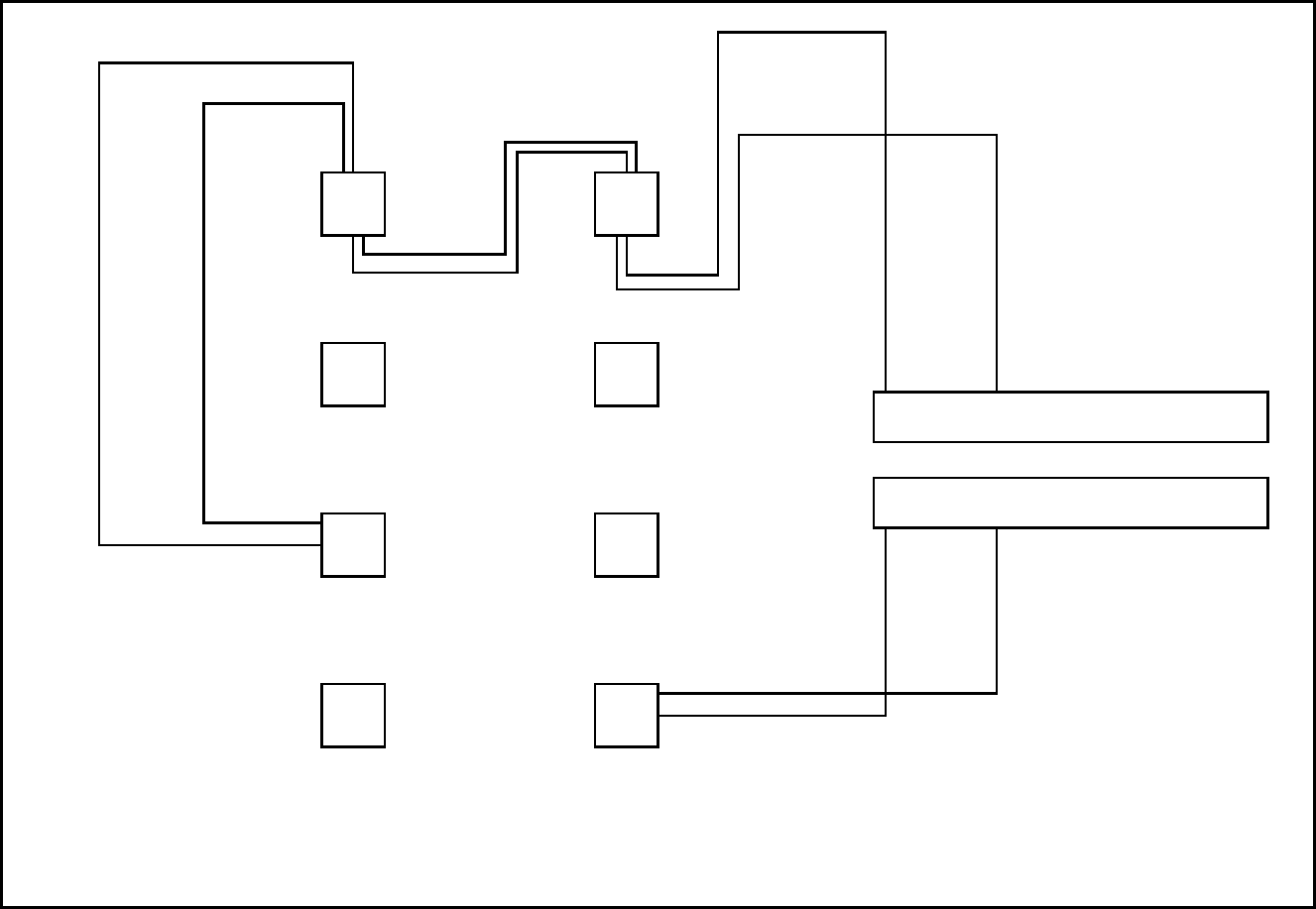}
\caption{Circles 15 and 21.}
\label{pc-7-6-15-21}
\end{subfigure}\hspace{.15cm}
\begin{subfigure}{0.23\textwidth}
\labellist
\small\hair 2pt
\pinlabel 20 at 30 30
\endlabellist
\includegraphics[width=\textwidth]{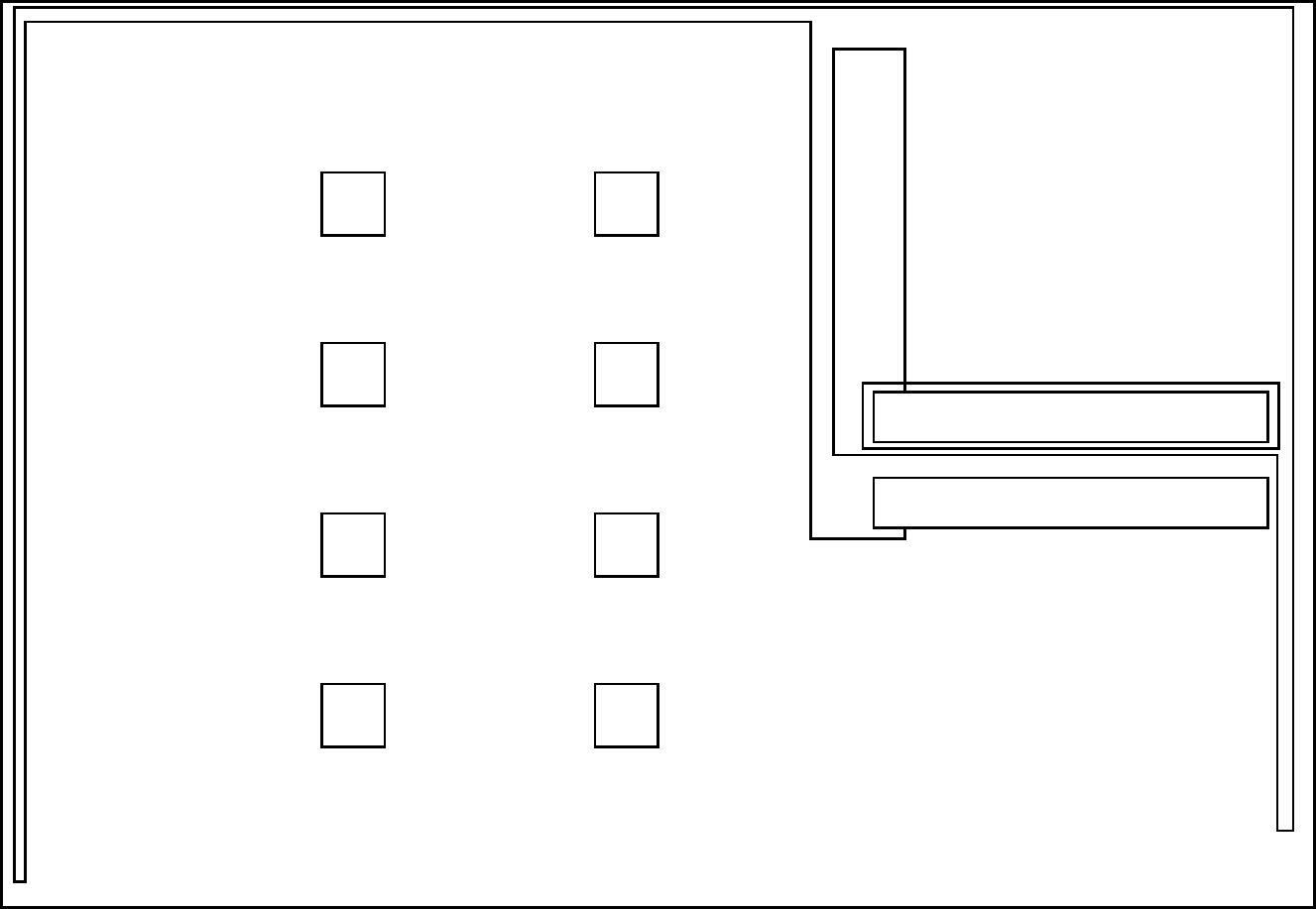}
\caption{Circles 20 and 22.}
\label{pc-7-6-20-22}
\end{subfigure}\hspace{.15cm}
\begin{subfigure}{0.23\textwidth}
\includegraphics[width=\textwidth]{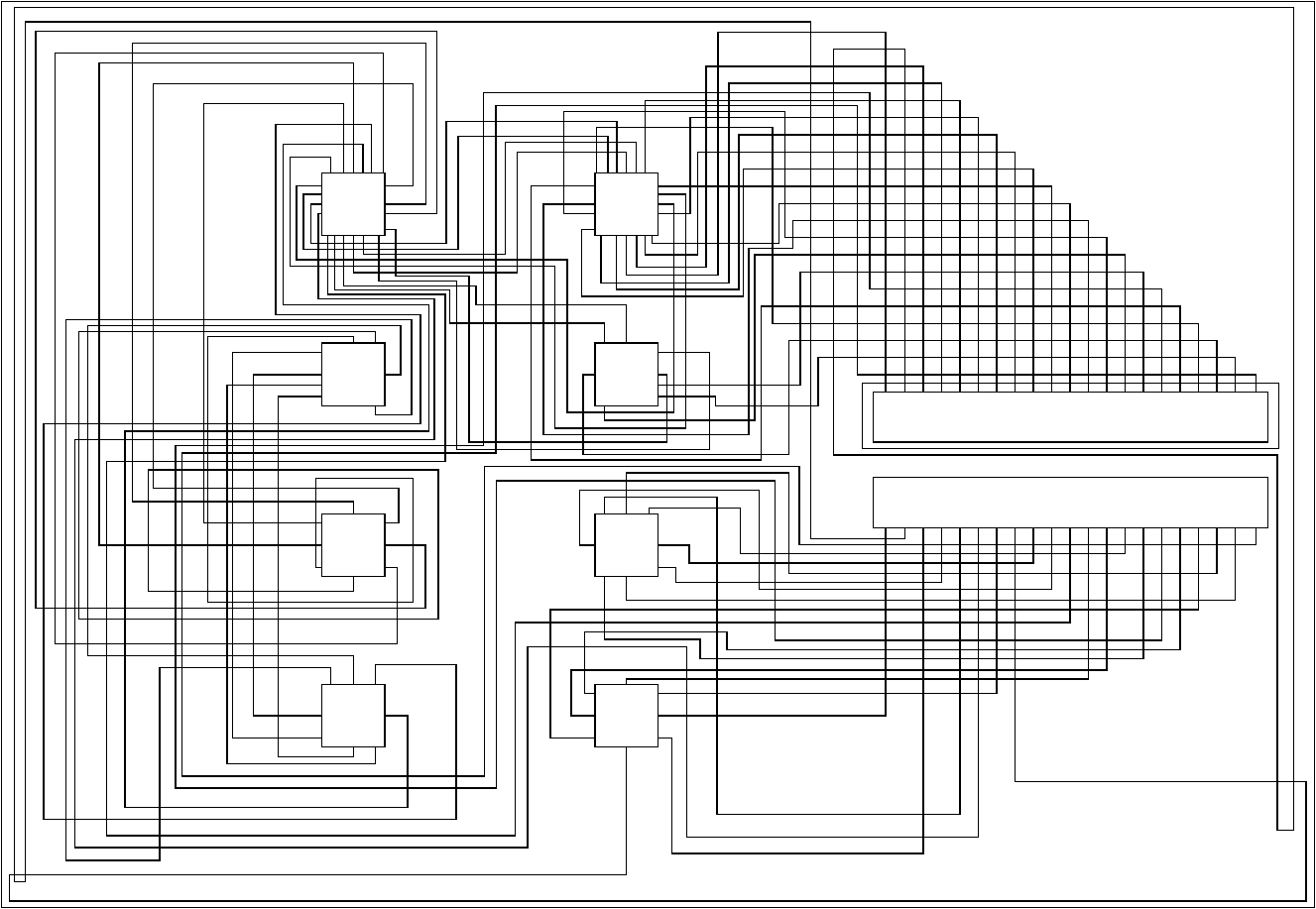}
\caption{All circles.}
\label{pc-7-6-all}
\end{subfigure}
	\caption{A pseudocoronation of $\7$. One of the two circles is labeled in each diagram except the last.\label{pc-7-6}}
\end{figure}

\begin{figure}
	\centering
	\begin{subfigure}{0.23\textwidth}
		\labellist
		\small\hair 2pt
		\pinlabel 6 at 300 195
		\endlabellist
		\includegraphics[width=\textwidth]{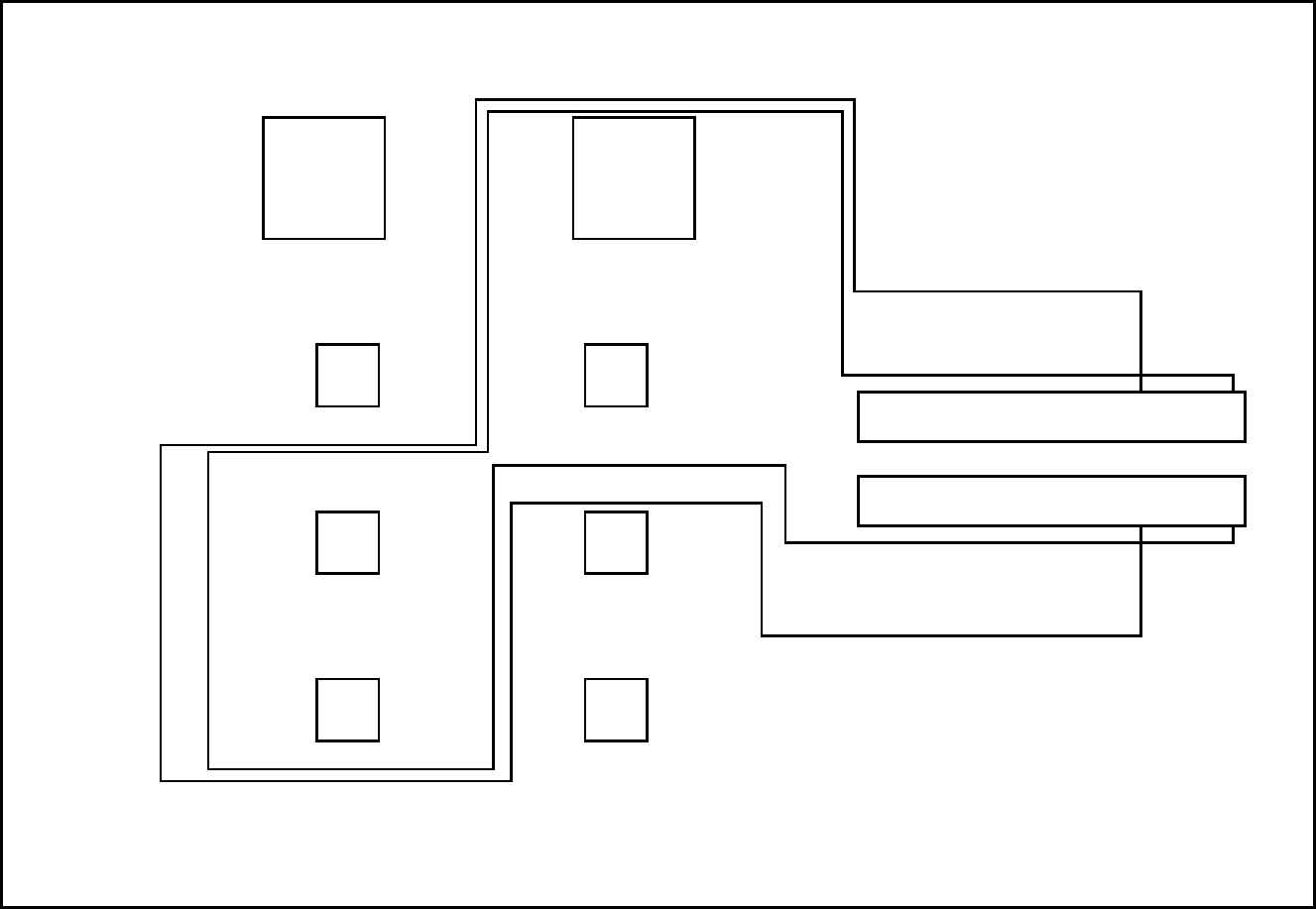}
		\caption{Circles 1 and 6.}
		\label{pc-10-133-1-6}
	\end{subfigure}\hspace{.15cm}
	\begin{subfigure}{0.23\textwidth}
		\labellist
		\small\hair 2pt
		\pinlabel 7 at 270 200
		\endlabellist
		\includegraphics[width=\textwidth]{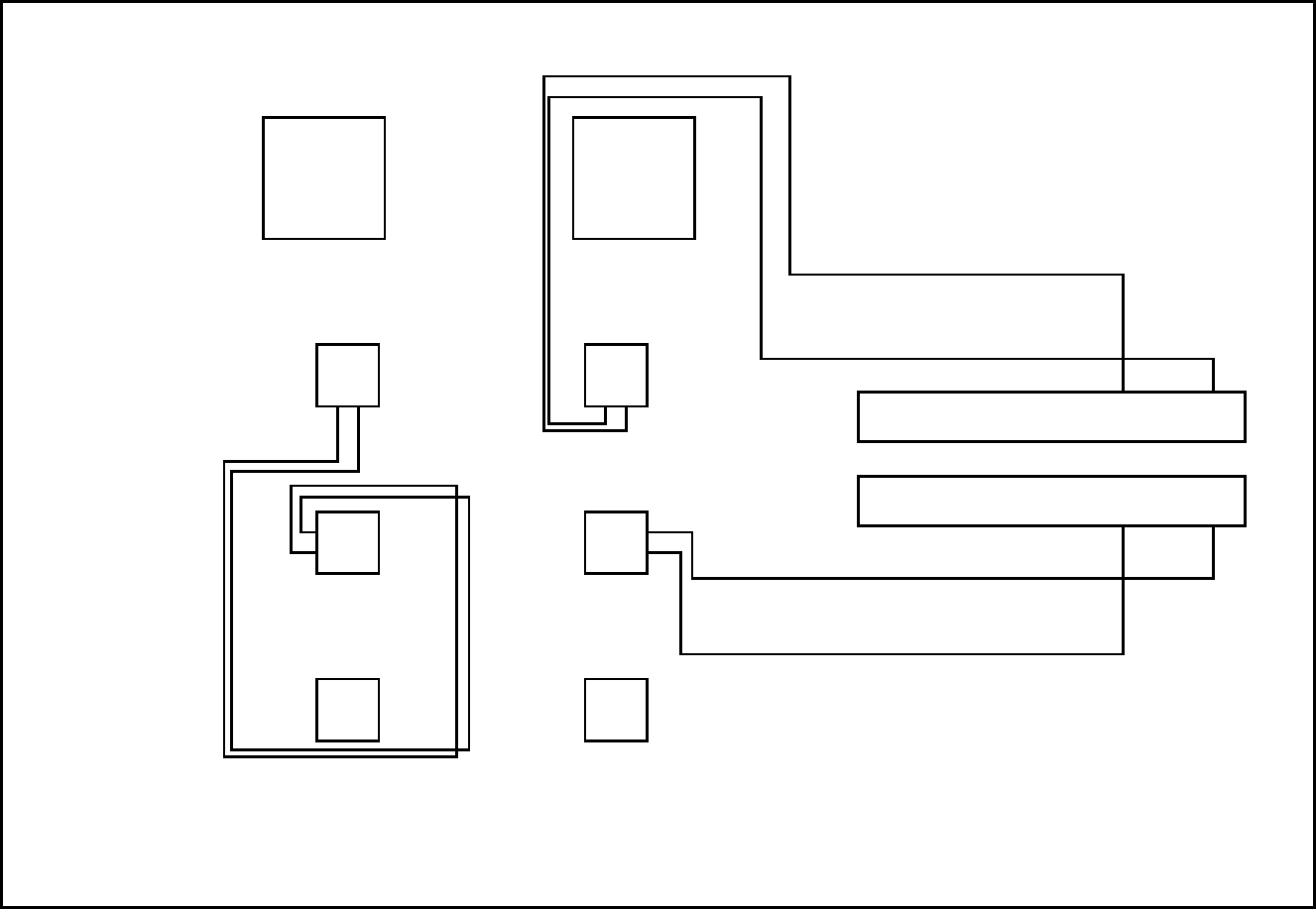}
		\caption{Circles 2 and 7.}
		\label{pc-10-133-2-7}
	\end{subfigure}\hspace{.15cm}
	\begin{subfigure}{0.23\textwidth}
		\labellist
		\small\hair 2pt
		\pinlabel 8 at 270 205
		\endlabellist
		\includegraphics[width=\textwidth]{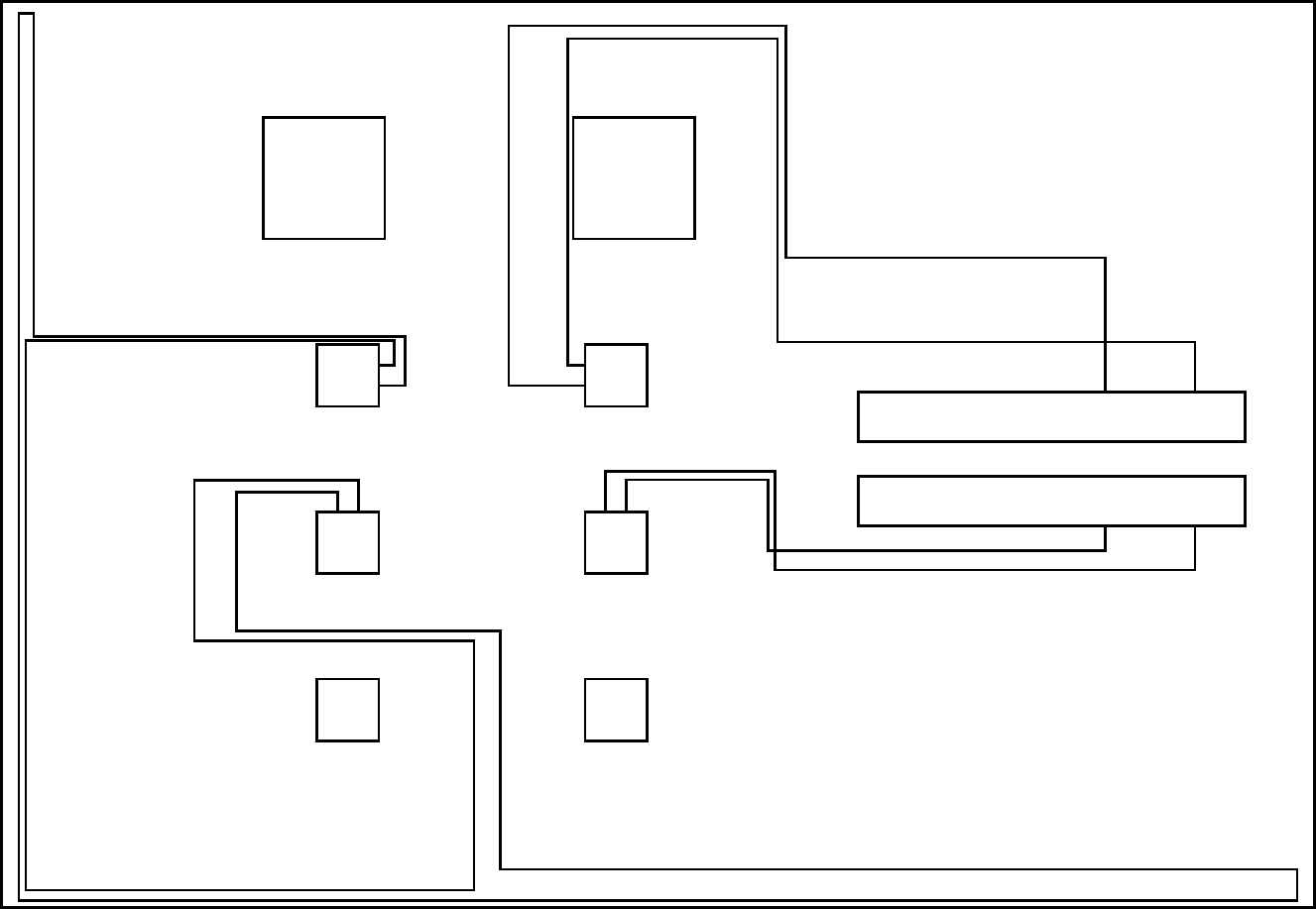}
		\caption{Circles 3 and 8.}
		\label{pc-10-133-3-8}
	\end{subfigure}\hspace{.15cm}
	\begin{subfigure}{0.23\textwidth}
		\labellist
		\small\hair 2pt
		\pinlabel 9 at 270 210
		\endlabellist
		\includegraphics[width=\textwidth]{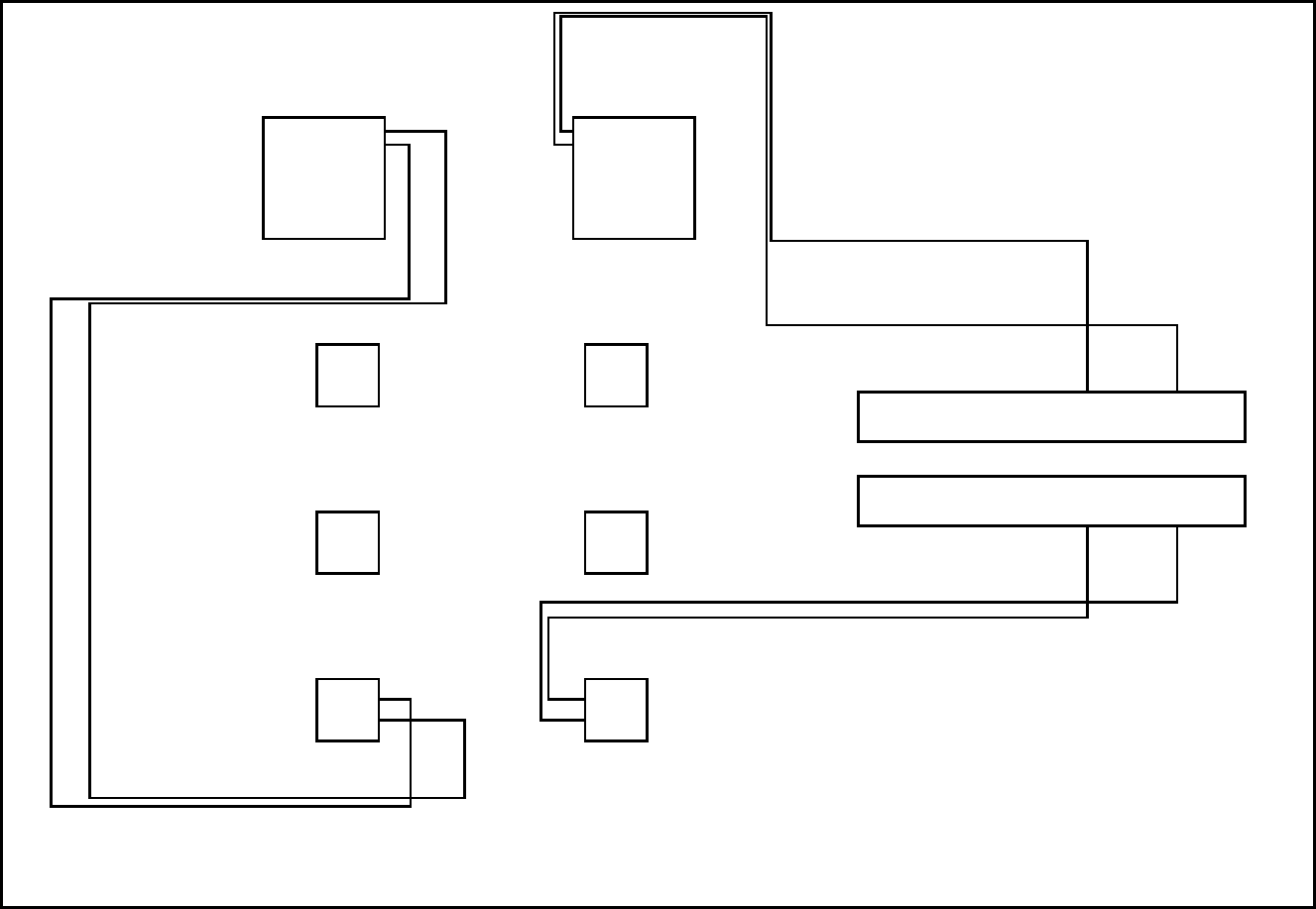}
		\caption{Circles 4 and 9.}
		\label{pc-10-133-4-9}
	\end{subfigure}\vspace{.25cm}
	\begin{subfigure}{0.23\textwidth}
		\labellist
		\small\hair 2pt
		\pinlabel 10 at 270 215
		\endlabellist
		\includegraphics[width=\textwidth]{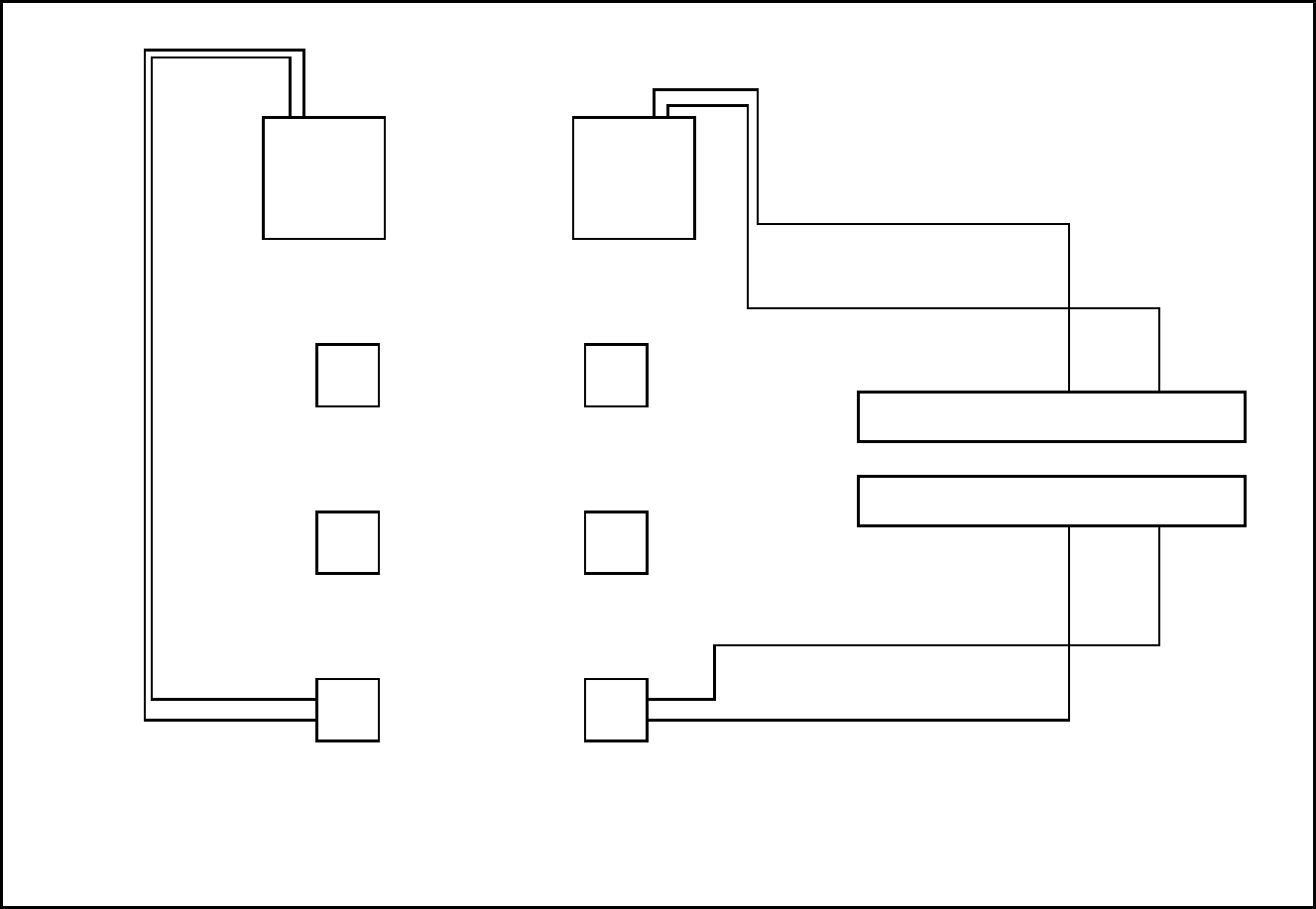}
		\caption{Circles 5 and 10.}
		\label{pc-10-133-5-10}
	\end{subfigure}\hspace{.15cm}
	\begin{subfigure}{0.23\textwidth}
		\labellist
		\small\hair 2pt
		\pinlabel 11 at 330 175
		\endlabellist
		\includegraphics[width=\textwidth]{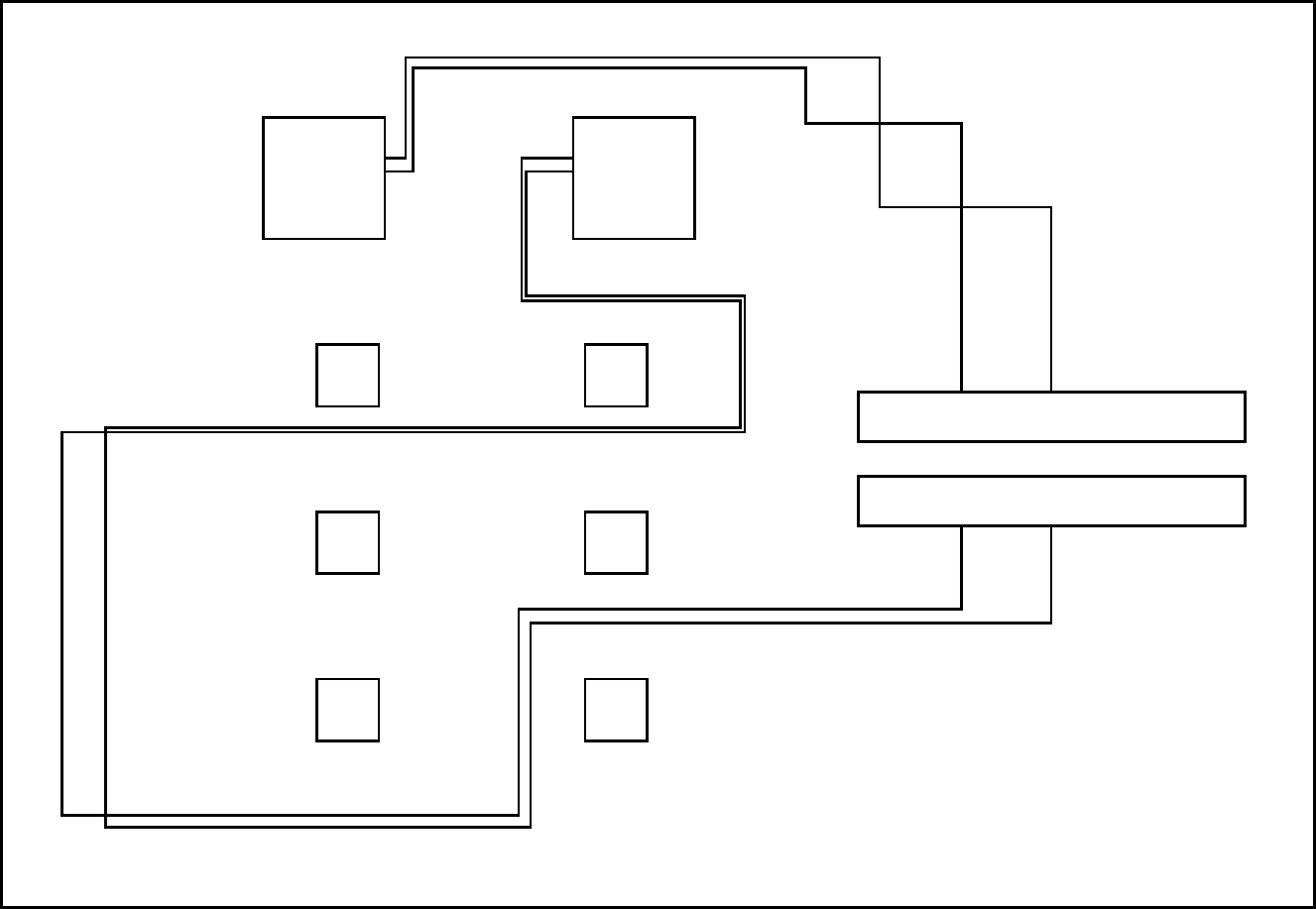}
		\caption{Circles 11 and 16.}
		\label{pc-10-133-11-16}
	\end{subfigure}\hspace{.15cm}
	\begin{subfigure}{0.23\textwidth}
		\labellist
		\small\hair 2pt
		\pinlabel 12 at 325 175
		\endlabellist
		\includegraphics[width=\textwidth]{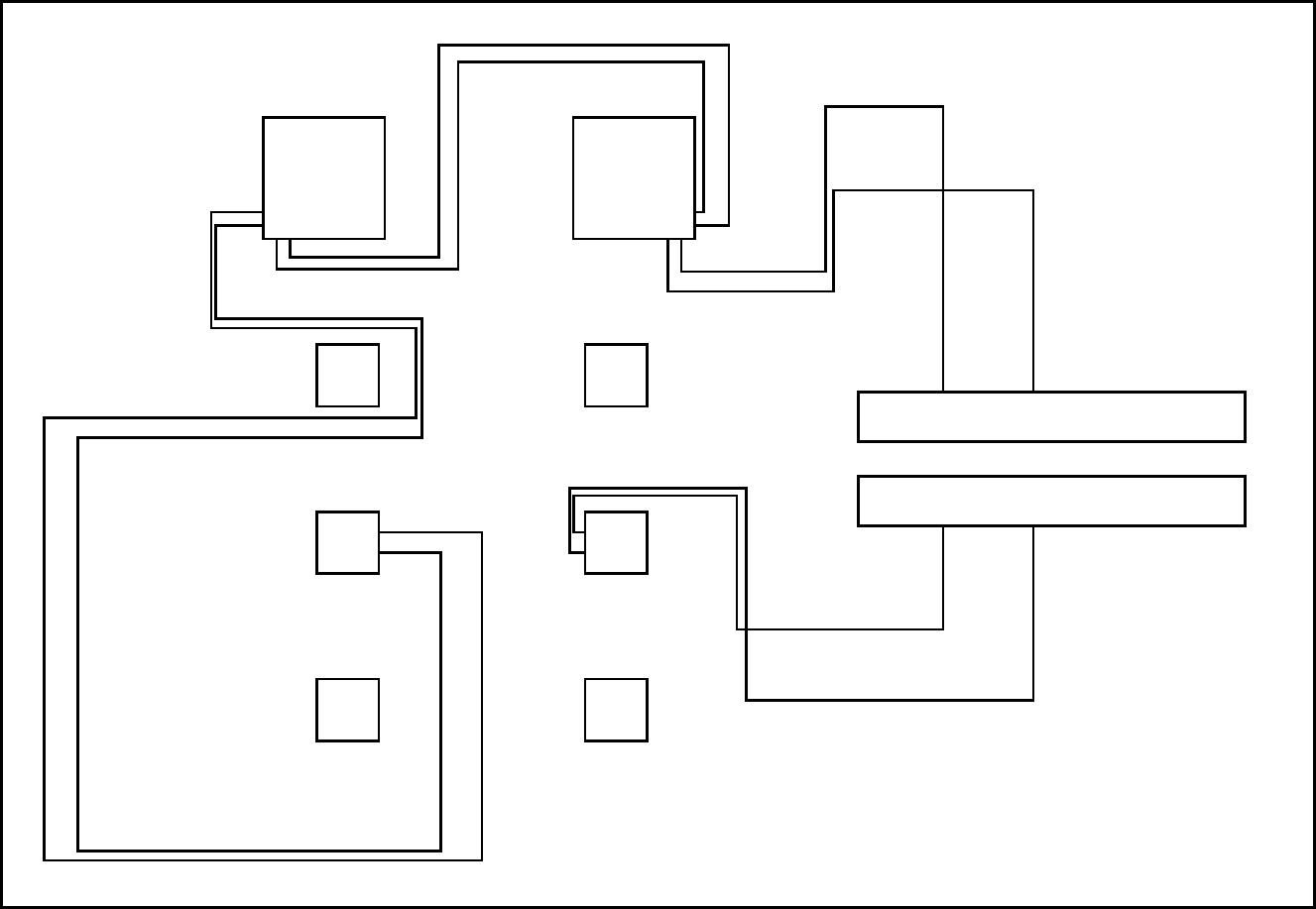}
		\caption{Circles 12 and 17.}
		\label{pc-10-133-12-17}
	\end{subfigure}\hspace{.15cm}
	\begin{subfigure}{0.23\textwidth}
		\labellist
		\small\hair 2pt
		\pinlabel 13 at 320 175
		\endlabellist
		\includegraphics[width=\textwidth]{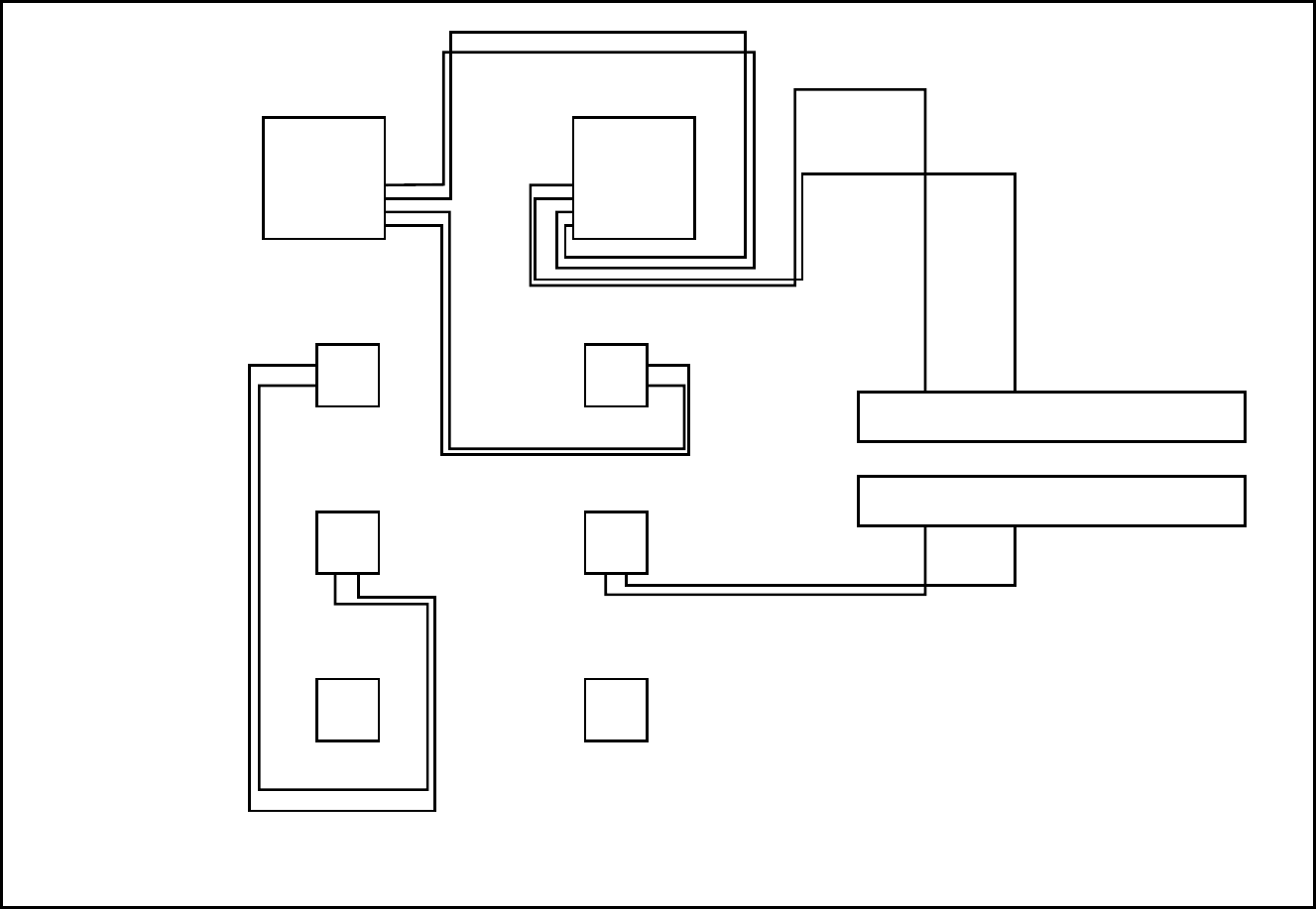}
		\caption{Circles 13 and 18.}
		\label{pc-10-133-13-18}
	\end{subfigure}\vspace{.25cm}
	\begin{subfigure}{0.23\textwidth}
		\labellist
		\small\hair 2pt
		\pinlabel 14 at 315 175
		\endlabellist
		\includegraphics[width=\textwidth]{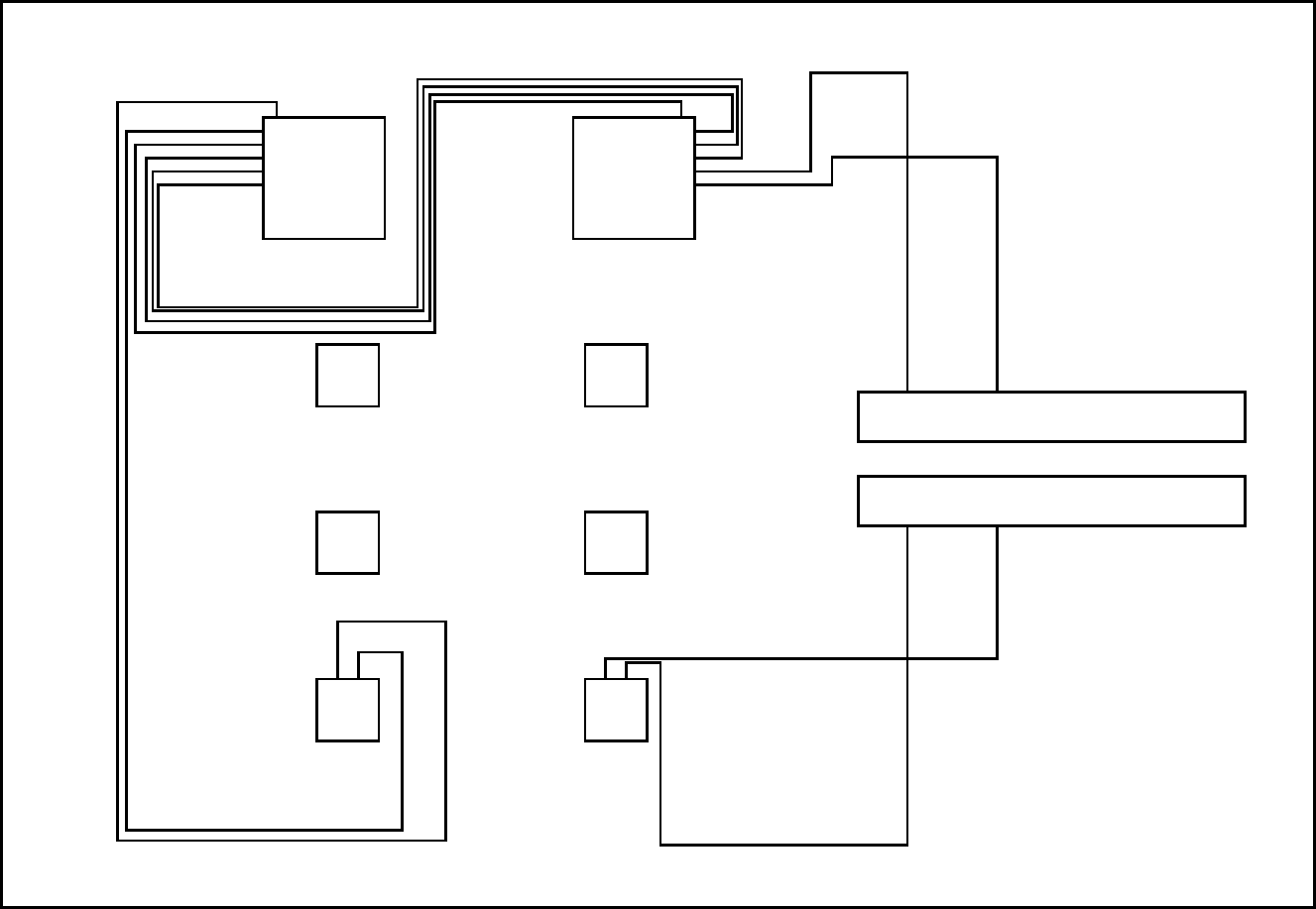}
		\caption{Circles 14 and 19.}
		\label{pc-10-133-14-19}
	\end{subfigure}\hspace{.15cm}
	\begin{subfigure}{0.23\textwidth}
		\labellist
		\small\hair 2pt
		\pinlabel 15 at 310 175
		\endlabellist
		\includegraphics[width=\textwidth]{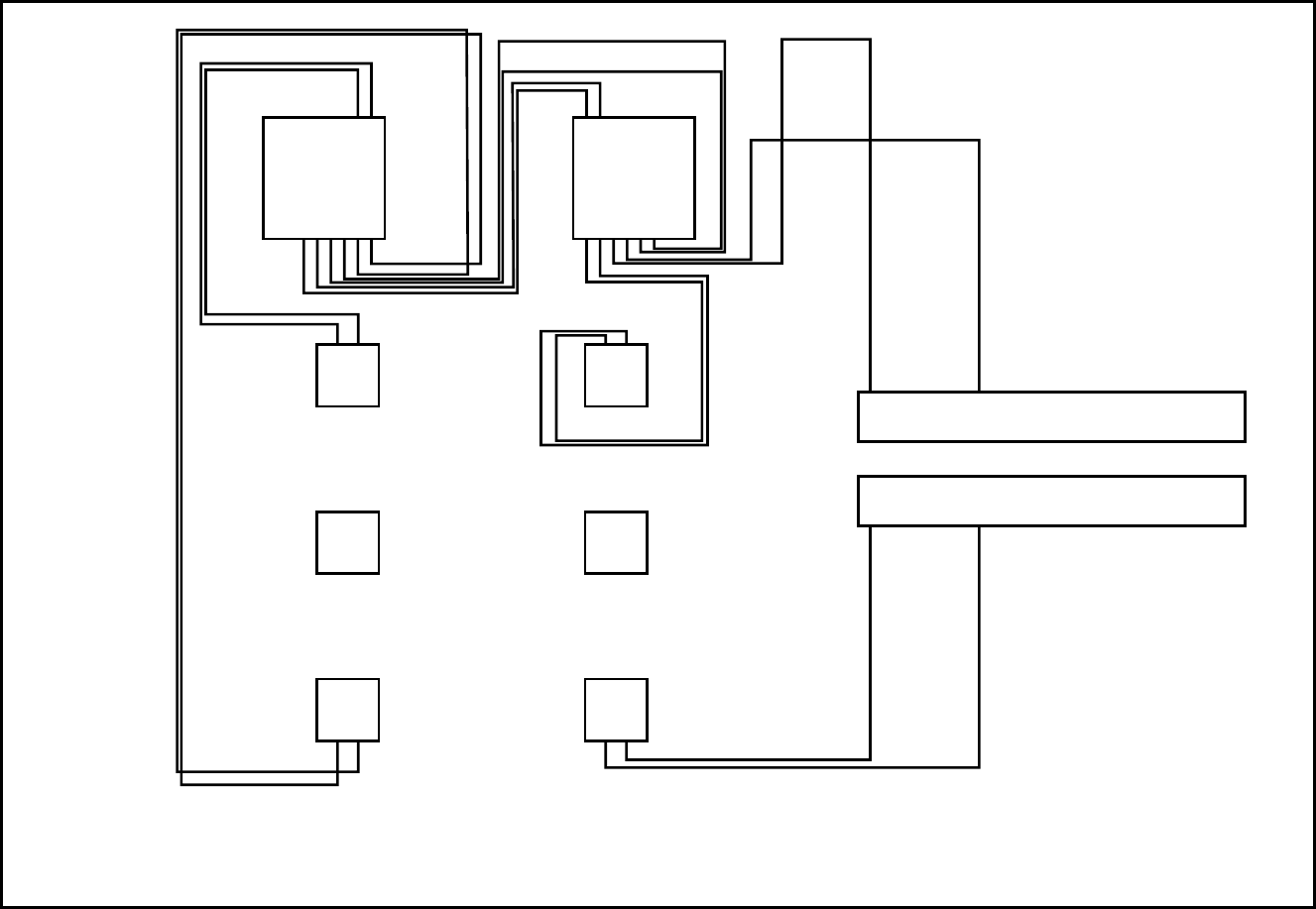}
		\caption{Circles 15 and 21.}
		\label{pc-10-133-15-21}
	\end{subfigure}\hspace{.15cm}
	\begin{subfigure}{0.23\textwidth}
		\labellist
		\small\hair 2pt
		\pinlabel 20 at 345 30
		\endlabellist
		\includegraphics[width=\textwidth]{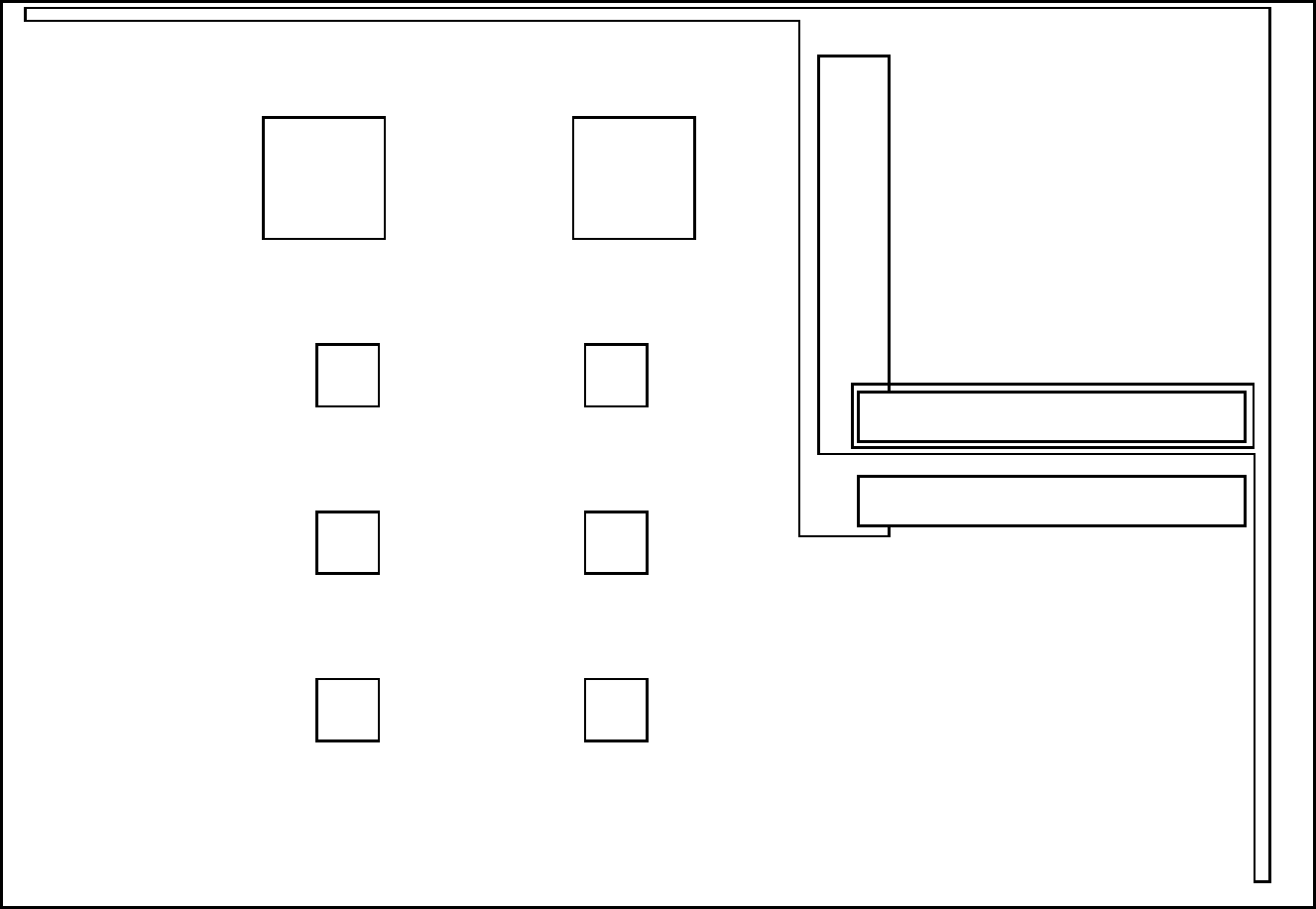}
		\caption{Circles 20 and 22.}
		\label{pc-10-133-20-22}
	\end{subfigure}\hspace{.15cm}
	\begin{subfigure}{0.23\textwidth}
		\includegraphics[width=\textwidth]{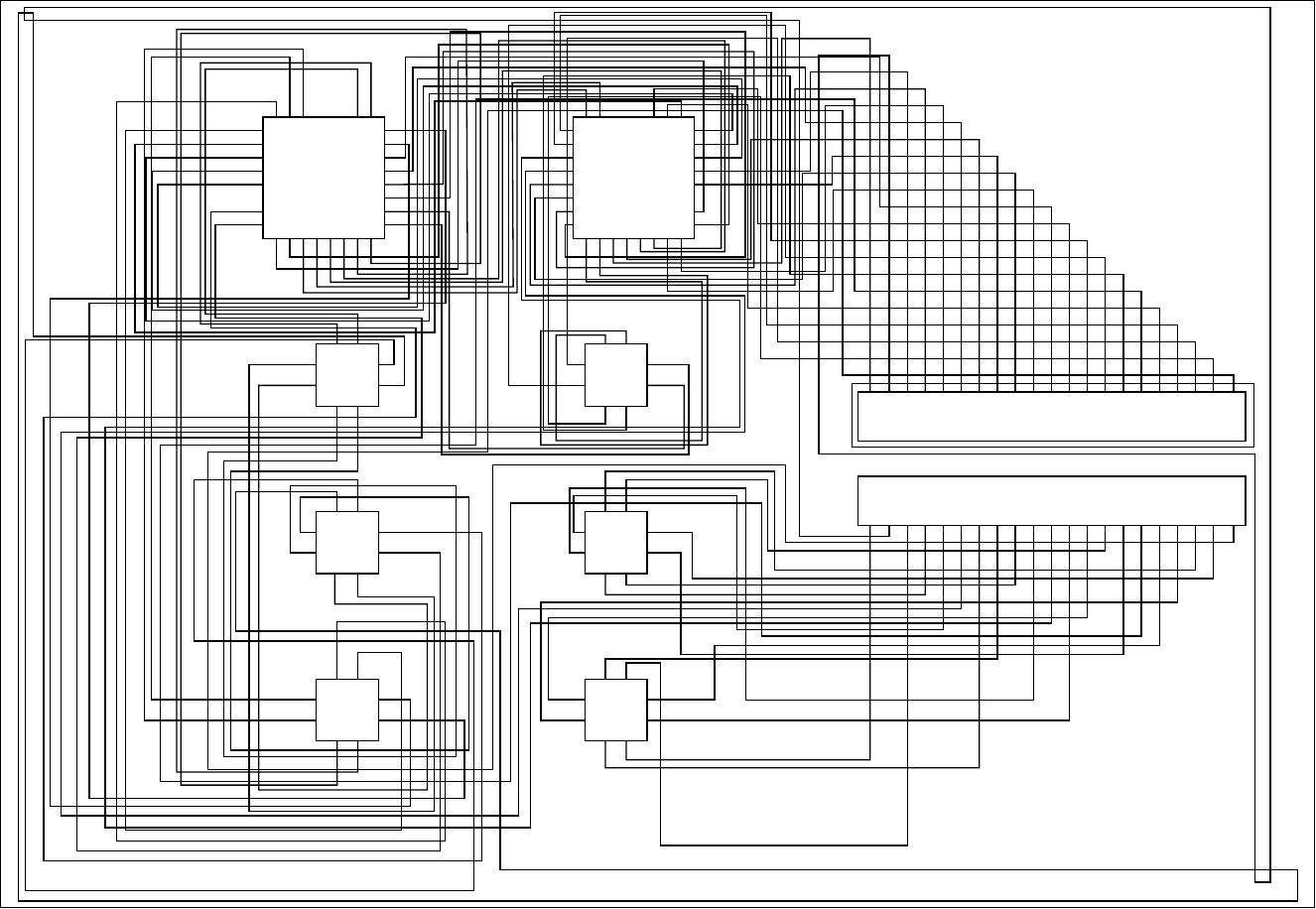}
		\caption{All circles.}
		\label{pc-10-133-all}
	\end{subfigure}
	\caption{A pseudocoronation of $\1$. One of the two circles is labeled in each diagram except the last.\label{pc-10-133}}
\end{figure}
\subsection{Equations and calculations}The number of regions in the pseudocoronations given by Figures \ref{pc-7-6} and \ref{pc-10-133} is large, and each region has multiple equations coming from the different choices of labeling. For this reason, there is software written by Umur \c{C}ift\c{c}i, called EQMaker, which produces the grading system as a text file, and there is a list of Mathematica commands for extracting from this system a new system of equations which expresses the salient set as a sum of formal parameters, all available at the \MYhref{http://www2.math.binghamton.edu/lib/exe/fetch.php/people/jwilliams/salientsets.zip}{author's website}. The grading systems for these diagrams both have integer solutions, the rank of the system for $\7$ is 21, and the rank of the system for $\1$ is 20. For this reason, the salient sets are not equal, so the corresponding smoothings are not isotopic.

\end{document}